\title{On saturation games}
\author{Dan Hefetz \thanks{School of Mathematics, University of Birmingham, Edgbaston,
Birmingham B15 2TT, United Kingdom. Email: d.hefetz@bham.ac.uk.
Research supported by EPSRC grant EP/K033379/1.} \and Michael
Krivelevich \thanks{School of Mathematical Sciences, Raymond and
Beverly Sackler Faculty of Exact Sciences, Tel Aviv University,
69978, Israel. Email: krivelev@post.tau.ac.il. Research supported in
part by USA-Israel BSF Grant 2010115 and by grant 912/12 from the
Israel Science Foundation.} \and Alon Naor
\thanks{School of Mathematical Sciences, Raymond and Beverly Sackler
Faculty of Exact Sciences, Tel Aviv University, Tel Aviv, 69978,
Israel. Email: alonnaor@post.tau.ac.il} \and Milo\v{s}
Stojakovi\'c\thanks{Department of Mathematics and Informatics,
University of Novi Sad, Serbia. Partly supported by Ministry of
Education and Science, Republic of Serbia, and Provincial
Secretariat for Science, Province of Vojvodina. Email:
milos.stojakovic@dmi.uns.ac.rs.}}
\documentclass[11pt]{article}
\usepackage{amsmath,amssymb,latexsym,color,epsfig,a4,enumerate}

\newif\ifnotesw\noteswtrue% T to show comments; F supresses.

\newtheorem{theorem}{Theorem}[section]
\newtheorem{lemma}[theorem]{Lemma}
\newtheorem{claim}[theorem]{Claim}

\newtheorem{observation}[theorem]{Observation}
\newtheorem{corollary}[theorem]{Corollary}
\newtheorem{conjecture}[theorem]{Conjecture}

\newtheorem{remark}[theorem]{Remark}

\newenvironment{proof}{\noindent{\bf Proof\,}}{\hfill$\Box$}

\begin{document}
\maketitle

\begin{abstract}
A graph $G = (V,E)$ is said to be \emph{saturated} with respect to a
monotone increasing graph property ${\mathcal P}$, if $G \notin
{\mathcal P}$ but $G \cup \{e\} \in {\mathcal P}$ for every $e \in
\binom{V}{2} \setminus E$. The \emph{saturation game} $(n, {\mathcal
P})$ is played as follows. Two players, called Mini and Max,
progressively build a graph $G \subseteq K_n$, which does not
satisfy ${\mathcal P}$. Starting with the empty graph on $n$
vertices, the two players take turns adding edges $e \in
\binom{V(K_n)}{2} \setminus E(G)$, for which $G \cup \{e\} \notin
{\mathcal P}$, until no such edge exists (i.e.\ until $G$ becomes
${\mathcal P}$-saturated), at which point the game is over. Max's
goal is to maximize the length of the game, whereas Mini aims to
minimize it. The \emph{score} of the game, denoted by $s(n,
{\mathcal P})$, is the number of edges in $G$ at the end of the
game, assuming both players follow their optimal strategies.

We
prove lower and upper bounds on the score of games in which the
property the players need to avoid is being $k$-connected, having
chromatic number at least $k$, and admitting a matching of a given
size. In doing so we demonstrate that the score of certain games can
be as large as the Tur\'an number or as low as the saturation number
of the respective graph property, and also that the score
might strongly depend on the identity of the first player to move.
\end{abstract}

\section{Introduction}
\label{sec::intro}

Let $n$ be a positive integer, let ${\mathcal P}$ be a monotone
increasing property of graphs on $n$ vertices and let $G = ([n],E)$
be a graph which does not satisfy ${\mathcal P}$. An edge $e \in
\binom{[n]}{2} \setminus E$ is called \emph{legal with respect to
$G$ and ${\mathcal P}$} if $G \cup \{e\} \notin {\mathcal P}$. A
graph $G = ([n],E)$ is said to be \emph{saturated with respect to
${\mathcal P}$} if $G \notin {\mathcal P}$ and there are no legal
edges with respect to $G$ and ${\mathcal P}$. Given a graph $H
\notin {\mathcal P}$ with vertex set $[n]$, the \emph{saturation
game} $(H, {\mathcal P})$ is played as follows. Two players, called
Mini and Max, progressively build a graph $G$, where $H \subseteq G
\subseteq K_n$, so that $G$ does not satisfy ${\mathcal P}$.
Starting with $G = H$, the two players take turns adding edges which
are legal with respect to the current graph $G$ and the property
${\mathcal P}$ until no such edge exists, at which point the game is
over. Max's goal is to maximize the length of the game, whereas Mini
aims to minimize it. The \emph{score} of the game, denoted by $s(H,
{\mathcal P})$, is the number of edges in $G$ at the end of the game
(recall that with some abuse of notation we use $G$ to denote the
graph built by both players at any point during the game) when both
players follow their optimal strategies. In fact, we would only be
interested in the case $H = \overline{K}_n$, where $\overline{K}_n$ is the
empty graph on $n$ vertices, but we generalize the definition of the
game for the purpose of simplifying the presentation of some of our
proofs. We abbreviate $s(\overline{K}_n, {\mathcal P})$ to $s(n,
{\mathcal P})$. Note that we did not specify which of the two
players starts the game. Since the score of a saturation game might
depend on this information, whenever studying a specific game we
will consider its score in two cases -- when Mini is the first player and when Max
is the first player. Where we do not explicitly specify the identity of
the first player, our related results hold in both cases.

Straightforward bounds on the score of a saturation game stem from the
corresponding \emph{saturation number} and \emph{Tur\'an number}.
Given a monotone increasing graph property ${\mathcal P}$, the
saturation number of ${\mathcal P}$, denoted by $sat(n, {\mathcal
P})$, is the minimum possible size of a saturated graph on $n$
vertices with respect to ${\mathcal P}$. Saturation numbers have
attracted a lot of attention since their introduction by Erd\H{o}s,
Hajnal and Moon~\cite{EHM}; many related results and open problems
can be found in the survey~\cite{FFS}. Similarly, the Tur\'an number
of ${\mathcal P}$, denoted by $ex(n, {\mathcal P})$, is the maximum
possible size of a saturated graph on $n$ vertices with respect to
${\mathcal P}$. The theory of Tur\'an numbers is a cornerstone of
Extremal Combinatorics; many related results and open problems can
be found e.g.\ in~\cite{BolBook}. It is immediate from the definition
of the saturation game $(n, {\mathcal P})$ that $sat(n, {\mathcal
P}) \leq s(n, {\mathcal P}) \leq ex(n, {\mathcal P})$.

Results on scores of saturation games are quite scarce. For example,
let $\mathcal{K}_3$ denote the property of containing a triangle. A
well-known theorem of Mantel (see e.g.~\cite{West}) asserts that
$ex(n, \mathcal {K}_3) = \lfloor n^2/4 \rfloor$. Moreover, since a
star is saturated with respect to $\mathcal {K}_3$ and, on the other
hand, no disconnected graph is, it follows that $sat(n, \mathcal
{K}_3) = n-1$ (this also follows from a more general result of
Erd\H{o}s, Hajnal and Moon~\cite{EHM}). In contrast to these exact
results, very little is known about $s(n, \mathcal {K}_3)$. The best
known lower bound, due to F\"uredi, Reimer and Seress~\cite{FRS}, is
of order $n \log n$. In the same paper, F\"uredi et al.\ attribute
an upper bound of $n^2/5$ to Erd\H{o}s; however, the proof is lost.
Bir\'{o}, Horn and Wildstrom~\cite{BHW} have recently improved the
upper bound of $\lfloor n^2/4 \rfloor$ which follows from Mantel's
Theorem to $\frac{26}{121}n^2+o\left(n^2\right)$. Additional
saturation-type games were recently studied in~\cite{CKOW}
and~\cite{PV}.

We begin our study of saturation games with games in which both
players are required to keep the connectivity of the graph below a
certain threshold. For every positive integer $k$ we would like to
determine $s(n, {\mathcal C}_k)$, where ${\mathcal C}_k$ is the
property of being $k$-vertex-connected and spanning. It is easy to
see that $ex(n, {\mathcal C}_k) = \binom{n-1}{2} + k - 1$ holds for
every positive integer $k \le n$. Very recently, it was shown
in~\cite{CKRW} that $s(n, {\mathcal C}) = \binom{n-2}{2} + 1$ for
every $n \geq 6$. Our first result shows that $s(n, {\mathcal C}_k)$
is almost as large as $ex(n, {\mathcal C}_k)$ for every fixed
positive integer $k$.

\begin{theorem} \label{th::allKconnectivity}
$s(n, {\mathcal C}_k) \geq \binom{n}{2} - 5k n^{3/2}$ for every
positive integer $k$ and sufficiently large $n$.
\end{theorem}

Using a different proof technique, for every $k \geq 5$ we can
improve the error term in the bound given in
Theorem~\ref{th::allKconnectivity}.

\begin{theorem} \label{th::Kconnectivity}
$s(n, {\mathcal C}_k) \geq \binom{n}{2} - (k-1)(2k-4)[n -
(k-1)(2k-3)]$ for every $k \geq 5$ and sufficiently large $n$.
\end{theorem}

\begin{remark} \label{rem::weakBounds}
The lower bounds on $s(n, {\mathcal C}_k)$ given in
Theorems~\ref{th::allKconnectivity} and~\ref{th::Kconnectivity}
are not as tight as the lower bound on $s(n, {\mathcal C})$
given in~\cite{CKRW}, which is matching the upper bound. Since ${\mathcal C}_k \subseteq {\mathcal C}$
for every $k \geq 1$, it may seem at first sight like $s(n, {\mathcal C}_k) \geq
s(n, {\mathcal C}) = \binom{n-2}{2} + 1$ should hold as well.
But as we will see later (see Remark~\ref{rem::monotonicity} below),
such an implication is not true in general.
\end{remark}

We now move on to study saturation games in which both players are
required to keep the chromatic number of the graph below
a certain threshold. For every integer $k \geq 2$ we would like to
determine $s(n, \chi_{> k})$, where $\chi_{> k}$ is the property of
having chromatic number at least $k+1$ (obviously $s(n, \chi_{> 1})
= 0$). It is easy to see that if $H$ is a graph on $n \geq k$
vertices which is saturated with respect to $\chi_{> k}$, then $H$
is complete $k$-partite. From this it easily follows that $sat(n,
\chi_{> k}) = (k-1)(n-1) - \binom{k-1}{2}$ and $ex(n, \chi_{> k}) =
\sum_{0 \leq i < j \leq k-1} \lfloor \frac{n+i}{k} \rfloor \cdot
\lfloor \frac{n+j}{k} \rfloor = \left(1 - 1/k + o(1) \right)
\binom{n}{2}$. Very recently, it was shown in~\cite{CKRW} that $s(n,
\chi_{> 2})$ is equal to the trivial upper bound, that is, $s(n,
\chi_{> 2}) = ex(n, \chi_{> 2}) = \lfloor n^2/4 \rfloor$.

Our first result regarding colorability games shows that, in
contrast to the $(n, \chi_{> 2})$ game, Mini does have a strategy to
ensure that $s(n, \chi_{> 3})$ is smaller than $ex(n, \chi_{> 3})$
by a non-negligible fraction.

\begin{theorem} \label{th::3colorability}
$s(n, \chi_{> 3}) \leq 21 n^2/64 + O(n)$.
\end{theorem}

Additionally, we prove that for every sufficiently large $k$, Max
has a strategy to ensure that $s(n, \chi_{> k})$ is not much smaller
than $ex(n, \chi_{> k})$.

\begin{theorem} \label{th::Kcolorability}
There exists a real number $C$ such that $s(n, \chi_{> k}) \geq
\left(1 - C \log k/k \right) \binom{n}{2}$ holds for every positive
integer $k$ and sufficiently large $n$.
\end{theorem}

Lastly, we study saturation games in which both players are required
to keep the size of every matching in the graph below a
certain threshold. Starting with the property $\mathcal{PM}$ of
admitting a perfect matching, it is easy to see that $ex(n,
\mathcal{PM}) = \binom{n-1}{2}$ for every even $n$. Moreover, using
Tutte's well-known necessary and sufficient condition for the
existence of a perfect matching~\cite{Tutte}, Mader~\cite{Mader}
characterized all graphs which are saturated with respect to
$\mathcal{PM}$. Using this characterization, it is not hard to show
that $sat(n, \mathcal{PM}) = \Theta(n^{3/2})$. We prove that $s(n,
\mathcal{PM})$ is almost as large as $ex(n, \mathcal{PM})$.

\begin{theorem} \label{th::perfectMatchingGame}
Let $n \geq 8$ be an even integer, then $s(n, \mathcal{PM}) \geq
\binom{n-4}{2}$.
\end{theorem}

We then move on to study $s(n, \mathcal{M}_k)$, where
$\mathcal{M}_k$ is the property of admitting a matching of size $k$
for some $k \le n/2$. It was proved by Erd\H{o}s and Gallai
in~\cite{EG} that
$$ ex(n, \mathcal{M}_k) = \max \left\{(k-1)(n-1) - \binom{k-1}{2},
\binom{2k-1}{2} \right\}. $$
Applying the Berge-Tutte
formula~\cite{BT}, Mader~\cite{Mader} also characterized all graphs
which are saturated with respect to $\mathcal{M}_k$, for every $1
\leq k \leq n/2$. Using this characterization, it is not hard to
derive that $sat(n, \mathcal{M}_k) = 3(k-1)$ if $k \leq n/3$,
$sat(n, \mathcal{M}_k) = \Theta(n^2/(n-2k))$ if $n/3 \leq k \leq n/2
- \sqrt{n}$ and $sat(n, \mathcal{M}_k) = \Theta(n^{3/2})$ if $n/2 -
\sqrt{n} \leq k \leq n/2$. Our next result shows that, at least when
$k$ is not too large with respect to $n$, the score $s(n,
\mathcal{M}_k)$ varies in order of magnitude, depending on the
parity of $k$ and the identity of the first player. This is in stark
contrast to all of our results mentioned until now (where changing the
identity of the first player might affect the score of the game,
but only by a negligible margin). Note that, among other results and
using different terminology, $s(n, \mathcal{M}_2)$ was determined
in~\cite{PV}.

\begin{theorem} \label{th::kMatching}
Let $k \geq 2$ be an integer. If Max is the first player and $k$ is
even, or Mini is the first player and $k$ is odd, then $s(n,
\mathcal{M}_k) \geq n-1$. In all other cases $s(n, \mathcal{M}_k)
\leq \binom{2k-1}{2}$.
\end{theorem}

\begin{remark} \label{rem::kMatching}
It follows from Theorem~\ref{th::kMatching} that that if $k$ is
fixed then, depending on the parity of $k$ and the identity of the
first player, either $s(n, \mathcal{M}_k) = \Theta(sat(n,
\mathcal{M}_k))$ or $s(n, \mathcal{M}_k) = \Theta(ex(n,
\mathcal{M}_k))$.
\end{remark}

\begin{remark} \label{rem::monotonicity}
It follows from Theorem~\ref{th::kMatching} that scores of saturation games
are not monotone in the following sense. There are monotone increasing
graph properties ${\mathcal P}_1$ and ${\mathcal P}_2$ such that
${\mathcal P}_1 \subseteq {\mathcal P}_2$ and yet $s(n, {\mathcal P}_1) < s(n, {\mathcal P}_2)$.
Also, there are monotone increasing graph properties ${\mathcal P}_1$ and ${\mathcal P}_2$
such that $sat(n, {\mathcal P}_1) < sat(n, {\mathcal P}_2)$ and $ex(n, {\mathcal P}_1) < ex(n, {\mathcal P}_2)$,
but $s(n, {\mathcal P}_1) > s(n, {\mathcal P}_2)$.
\end{remark}

\subsection{Notation and preliminaries}

\noindent For the sake of simplicity and clarity of presentation, we
do not make a par\-ti\-cu\-lar effort to optimize some of the
constants obtained in our proofs. We also omit floor and ceiling
signs whenever these are not crucial. Throughout the paper, $\log$
stands for the natural logarithm. We say that a graph property ${\mathcal P}$ holds
\emph{asymptotically almost surely}, or a.a.s.\ for brevity, if the
probability of satisfying ${\mathcal P}$ tends to 1 as the number of
vertices $n$ tends to infinity. Our graph-theoretic notation is
standard and follows that of~\cite{West}. In particular, we use the
following.

For a graph $G$, let $V(G)$ and $E(G)$ denote its sets of vertices
and edges respectively, and let $v(G) = |V(G)|$ and $e(G) = |E(G)|$.
For a set $U \subseteq V(G)$ and a vertex $w \in V(G)$, let $N_G(w,
U) = \{u \in U : wu \in E(G)\}$ denote the set of neighbors of $w$
in $U$ and let $d_G(w, U) = |N_G(w, U)|$. For disjoint sets $U, W
\subseteq V(G)$ let $N_G(W,U) = \bigcup_{w \in W} N_G(w, U)$. We
abbreviate $N_G(w, V(G))$ to $N_G(w)$, and $N_G(W, V(G) \setminus
W)$ to $N_G(W)$. The minimum degree of a graph $G$ is denoted by
$\delta(G)$. Often, when there is no risk of confusion, we omit the
subscript $G$ from the notation above. For a set $S \subseteq V(G)$,
let $G[S]$ denote the subgraph of $G$ induced by the vertices of
$S$. A connected component $C$ of a graph $G$ is said to be
\emph{non-trivial} if it contains an edge. The size of a maximum
matching in a graph $G$ is denoted by $\nu(G)$.

Assume that some saturation game $(H, {\mathcal P})$ is in progress,
where $H$ is a graph on $n$ vertices. The edges of $K_n \setminus G$
are called \emph{free} (recall that at any point during the game, we
use $G$ to denote the graph built by both players up to that point).
A \emph{round} of the game consists of a move by the first player
and a counter move by the second player. We say that a player
follows the \emph{trivial strategy} if in every move he claims an
arbitrary legal edge.

We end this subsection by proving the following lemma which asserts
that, without any saturation restrictions, either player can build a
long path that includes all vertices of positive degree. This lemma
will be useful for the connectivity and the matching games we will
study. We refer to the strategy described in the proof of the lemma
as \emph{the long path strategy}.

\begin{lemma} \label{lem::longPath}
Let $n \geq 3$ and $1 \leq \ell \leq n - 2$ be integers. Then
starting with the empty graph on $n$ vertices, either player can
ensure that, immediately after his $i$th move for some $i$, the
graph $G$ will contain a path $P$ such that the following three
properties are satisfied:
\begin{description}
\item [(a)] The length of $P$ is either $\ell$ or $\ell + 1$;
\item [(b)] If $u \in V(G) \setminus V(P)$, then $d_G(u) = 0$;
\item [(c)] At least one of the endpoints of $P$ has degree one in $G$.
\end{description}
\end{lemma}

\begin{proof}
We prove our claim by induction on $\ell$. For convenience we denote
the player who wishes to build the path $P$ by ${\mathcal A}$ and
the other player by ${\mathcal B}$. For $\ell = 1$ the correctness
of our claim is obvious as, in his first move, ${\mathcal A}$ can
build a path of length 1 if he is the first player and of length 2
otherwise. Assume our claim holds for some $1 \leq \ell < n - 2$; we
will prove it holds for $\ell + 1$ as well. First, ${\mathcal A}$
builds a path $P_{\ell}$ which satisfies properties (a), (b) and (c)
for $\ell$; the induction hypothesis ensures that ${\mathcal A}$ has
a strategy to do so. If $P_{\ell}$ is of length $\ell+1$ then there
is nothing to prove, so assume $P_{\ell}$ is of length $\ell$. Let
$P_{\ell} = (u_0, \ldots, u_{\ell})$ and assume without loss of
generality that $d_G(u_0) = 1$. Let $xy$ denote the edge ${\mathcal
B}$ claims in his subsequent move. We distinguish between the
following four cases:
\begin{description}
\item [(1)] If $\{x,y\} \subseteq V(P_{\ell})$, then ${\mathcal A}$ claims $u_{\ell} z$ for some isolated vertex $z$.
\item [(2)] If $\{x,y\} \cap V(P_{\ell}) = \emptyset$, then ${\mathcal A}$ claims $u_{\ell} x$.
\item [(3)] If $x \in \{u_0, u_{\ell}\}$ and $y \notin V(P_{\ell})$, then ${\mathcal A}$ claims $y z$ for some isolated vertex $z$.
\item [(4)] If $x \in V(P_{\ell}) \setminus \{u_0, u_{\ell}\}$ and $y \notin V(P_{\ell})$, then ${\mathcal A}$ claims $u_{\ell} y$.
\end{description}
It is easy to see that in all of the four cases above, ${\mathcal
A}$ can follow the proposed strategy and, by doing so, he builds a
path which satisfies conditions (a), (b) and (c) for $\ell + 1$.
\end{proof}

The rest of this paper is organized as follows. In
Section~\ref{sec::connectivity} we prove
Theorems~\ref{th::allKconnectivity} and~\ref{th::Kconnectivity}. In
Section~\ref{sec::color} we prove Theorems~\ref{th::3colorability}
and~\ref{th::Kcolorability}. In Section~\ref{sec::matching} we prove
Theorems~\ref{th::perfectMatchingGame} and~\ref{th::kMatching}.
Finally, in Section~\ref{sec::openprob} we present some open
problems.

\section{Connectivity games} \label{sec::connectivity}

In this section we study connectivity games, that is, saturation
games in which both players are required to keep the connectivity of
the graph below a certain threshold.

\textbf{Proof of Theorem~\ref{th::allKconnectivity}} Since $s(n,
{\mathcal C}_1)$ was determined in~\cite{CKRW}, we can assume that
$k \geq 2$. We present a strategy for Max; it is divided into the
following three stages:

\begin{description}
\item [Stage I:] Max follows the long path strategy until $G$ contains a path $P = (u_0, \ldots, u_{\ell})$ of length $\ell \in \{n - k - \lceil\sqrt{n}\rceil - 1, n - k - \lceil\sqrt{n}\rceil\}$ which includes all vertices of positive degree in $G$. He then proceeds to Stage II.
\item [Stage II:] Let $t$ and $r$ be the unique integers satisfying $\ell + 1 = \lceil 4 \sqrt{n} \rceil t + r$ and $0 \leq r < \lceil 4 \sqrt{n} \rceil$. Let $v_1, \ldots, v_k$ be $k$ arbitrary vertices of $V(G) \setminus V(P)$ and let $F = \{v_i u_{\lceil 4 \sqrt{n} \rceil j} : 1 \leq i \leq k, \; 1 \leq j \leq t\}$. In each of his moves in this stage, Max claims an arbitrary free edge of $F$. Once no such edge exists, this stage is over and Max proceeds to Stage III.
\item [Stage III:] Throughout this stage, Max follows the trivial strategy.
\end{description}

Our first goal is to prove that Max can indeed follow the proposed
strategy. This is obvious for Stage III and follows for Stage I by
Lemma~\ref{lem::longPath} (note that there are isolated vertices at
the end of Stage I so $G$ is certainly not $k$-connected at that
point). Since every vertex of $V(G) \setminus V(P)$ is isolated at
the end of Stage I and at most $2|F|+1 \leq 2k \sqrt{n}/4 < k (k +
\sqrt{n} - 1)/2$ edges are claimed by both players throughout Stage
II, it follows that, at the end of Stage II, there exists a vertex
$u \in V(G) \setminus V(P)$ such that $d_G(u) < k$; in particular,
$G$ is not $k$-connected at that point. Hence, Max can follow Stage
II of the proposed strategy.

At the end of the game, $G$ is saturated with respect to
$k$-connectivity, that is, $G$ is not $k$-connected but $G + uv$ is
$k$-connected for every $u,v \in V(G)$ such that $uv \notin E(G)$.
Let $S \subseteq V(G)$ be a cut set of $G$ of size $k-1$. Since $|S|
< k$, it follows that $\{v_1, \ldots, v_k\} \setminus S \neq
\emptyset$; assume without loss of generality that $v_1 \notin S$.
Let $A$ denote the connected component of $G \setminus S$ which
contains $v_1$ and let $B = V(G) \setminus (A \cup S)$. We claim
that $|B| \leq 5k \sqrt{n}$. This is obvious if $B \subseteq (V(G)
\setminus V(P)) \cup \{u_0, \ldots, u_{\lceil 4 \sqrt{n} \rceil -
1}, u_{\lceil 4 \sqrt{n} \rceil t + 1}, \ldots, u_{\ell}\}$. Assume
then that there exists a vertex $u_i \in B \cap \{u_{\lceil 4
\sqrt{n} \rceil}, \ldots, u_{\lceil 4 \sqrt{n} \rceil t}\}$; let $1
\leq j < t$ denote the unique index such that $\lceil 4 \sqrt{n}
\rceil j < i < \lceil 4 \sqrt{n} \rceil (j+1)$ (note that, for every
$1 \leq j \leq t$, we have $u_{\lceil 4 \sqrt{n} \rceil j} \notin B$
as $v_1 u_{\lceil 4 \sqrt{n} \rceil j} \in E(G)$ holds by Stage II
of the proposed strategy). Since $u_i \in B$, it follows that there
is no path between $u_i$ and $v_1$ in $G \setminus S$. It follows
that $|S \cap \{u_{\lceil 4 \sqrt{n} \rceil j}, \ldots, u_{\lceil 4
\sqrt{n} \rceil (j+1)}\}| \geq 2$. Since this is true for every
vertex of $B \cap \{u_{\lceil 4 \sqrt{n} \rceil}, \ldots, u_{\lceil
4 \sqrt{n} \rceil t}\}$, it follows that $B \cap \{u_{\lceil 4
\sqrt{n} \rceil}, \ldots, u_{\lceil 4 \sqrt{n} \rceil t}\}$ is the
union of at most $|S|-1$ subpaths of $P$, each of length at most $4
\sqrt{n}$. We conclude that $|B| \leq 5k \sqrt{n}$ as claimed.
Since, as noted above, $G$ is saturated with respect to
$k$-connectivity, it follows that $xy \notin E(G)$ if and only if $x
\in A$ and $y \in B$ (or vice versa). Hence $e(G) = \binom{n}{2} -
|A||B| \geq \binom{n}{2} - 5k \sqrt{n} (n - k + 1 - 5k \sqrt{n})
\geq \binom{n}{2} - 5k n^{3/2}$ as claimed. {\hfill $\Box$
\medskip\\}

\textbf{Proof of Theorem~\ref{th::Kconnectivity}} We present a
strategy for Max; it is divided into the following three stages:
\begin{description}
\item [Stage I:] Let $r = \frac{n}{2k-3}$, let $t = n - r$, let
$V_0 = \{v_1, \ldots, v_r\}$ be a subset of $V(G)$ and let $V(G)
\setminus V_0 = \{u_1, \ldots, u_t\}$. Max's goal in this stage is
to ensure that for every set $B \subseteq V \setminus V_0$, by the
end of the stage $|N_G(B, V_0)| \geq |B|r/t$ will hold. He does so
in the following way. In each of his moves in this stage, Max claims
$u_i v_{\lceil i r/t \rceil}$ where $i$ is the smallest positive
integer for which $u_i v_{\lceil i r/t \rceil}$ is free. As soon as
all edges of $\{u_i v_{\lceil i r/t \rceil} : 1 \leq i \leq t\}$ are
claimed, Max proceeds to Stage II.
\item [Stage II:] Let $H$ be a $k$-connected graph on $r$ vertices such
that $e(H)$ is minimal among all such graphs. Max ensures that $V_0$
will contain a copy of $H$ and then proceeds to Stage III.
\item [Stage III:] Throughout this stage, Max follows the trivial strategy.
\end{description}

Our first goal is to prove that Max can indeed follow the proposed
strategy. This is obvious for Stage III. For Stages I and II this
follows since at most $2 (t + \lceil k r/2\rceil) < k n/2$ edges are
claimed by both players during these two stages, where the
inequality follows by the definition of $r$ since $k \geq 5$. Since
there is no $k$-connected graph on $n$ vertices and strictly less
than $kn/2$ edges, it follows that Max can claim any free edge
throughout Stages I and II.

At the end of the game, $G$ is saturated with respect to
$k$-connectivity, that is, $G$ is not $k$-connected but $G + uv$ is
$k$-connected for every $u,v \in V(G)$ such that $uv \notin E(G)$.
Let $S \subseteq V(G)$ be a cut set of $G$ of size $k-1$. It follows
by Stage II of Max's strategy that $V_0 \setminus S$ is contained in
one connected component of $G \setminus S$; let $A$ denote this
component and let $B = V(G) \setminus (A \cup S)$. We claim that
$|B| \leq t(k-1)/r$. Indeed, suppose for a contradiction that $|B| >
t(k-1)/r$. It follows by Stage I of Max's strategy that $|N_G(B,
V_0)| \geq \lceil|B| r/t \rceil\geq k$. Since $|S| < k$ it follows
that $N_G(B, V_0 \setminus S) \neq \emptyset$ contrary to $S$ being
a cut set. Since, as noted above, $G$ is saturated with respect to
$k$-connectivity, it follows that $xy \notin E(G)$ if and only if $x
\in A$ and $y \in B$ (or vice versa). Hence $e(G) = \binom{n}{2} -
|A||B| \geq \binom{n}{2} - \frac{t(k-1)}{r} \left(n-k+1 -
\frac{t(k-1)}{r}\right) = \binom{n}{2} - (k-1)(2k-4)[n -
(k-1)(2k-3)]$ as claimed. {\hfill $\Box$ \medskip\\}

\section{Colorability games} \label{sec::color}

In this section we study colorability games, that is, saturation
games in which both players are required to keep the chromatic
number of the graph below a certain threshold.

\textbf{Proof of Theorem~\ref{th::3colorability}} Since the proof is
quite technical, even though it is based on a very simple idea, we begin
by briefly describing this idea. Regardless of Mini's strategy, at the
end of the game $G$ will be a complete 3-partite graph. Since she would
like to minimize the number of its edges, she should try to unbalance
its parts. She will do so by making sure one part is small, namely, its
size is at most $\lceil n/4 \rceil$. This will be achieved by connecting (by her and
Max's edges) an arbitrary vertex $v_0$ to roughly $3n/4$ vertices. In order to eventually prove Mini can
achieve this, we will show that for every vertex $x$ Mini cannot connect to $v_0$,
Max must have ``used up'' at least 3 of his moves.

We first introduce
some notation and terminology that will be used throughout this
proof. Let $v_0 \in V(G)$ be an arbitrary vertex. This vertex
determines the following partition $V(G) = T_G \cup M_G \cup B_G$:
$T_G$ consists of all vertices which receive the same color as $v_0$
in \emph{any} proper 3-coloring of $G$, $M_G = N_G(T_G)$ and $B_G =
V(G) \setminus (T_G \cup M_G)$. In particular, $T_G = \{v_0\}$,
$M_G = \emptyset$ and $B_G = V(G) \setminus \{v_0\}$ hold
before the game starts. The vertices of $T_G$ are called \emph{top}
vertices, the vertices of $M_G$ are called \emph{middle} vertices and
the vertices of $B_G$ are called \emph{bottom} vertices.
For any vertex $u \in V(G)$, let $\Gamma_G(u)$ denote the connected
component of $G[\{u\} \cup M_G]$ which contains $u$.
A connected component of $G[M_G]$ is called a
\emph{middle-component}. Note that at any point during the game,
every middle-component is 2-colorable. If $u \in B_G$, $v \in M_G$, $uv
\in E(G)$ and $C$ is the middle-component containing $v$, then $u$
is said to be \emph{attached} to $C$. When there is no risk of confusion,
we omit the subscript $G$ from the above notation.

Note that if
$x$ is a top (respectively middle) vertex, then $x$ remains a top
(respectively middle) vertex throughout the game. On the other hand,
if $x \in B$ and some player claims $xy$ for some top vertex $y$,
then $x$ is \emph{moved} to the middle, that is, $x$ becomes a middle
vertex. Moreover, if $x \in B$ and some player claims $xy$ for some
vertex $y \in M \cup B$, then either $x$ remains in $B$ or it is
moved to the top.

During the game, Mini will want $G$ to satisfy certain structural properties.
In order to describe these we introduce some more definitions,
starting with the following two properties of a given graph $G$ on $n$ vertices with
the partition $V(G) = T \cup M \cup B$:

\begin{description}
\item [(P1)] $E(G[B]) = \emptyset$;
\item [(P2)] Every middle-component of $G$ has at most one attached vertex.
\end{description}

Next, we define the set of bad edges with respect to $G$ as follows:
$$
BAD_G := \{uv \in \binom{V(G)}{2}: \exists x,y \in B, \textrm{ } x \neq y, \textrm{ } u \in V(\Gamma(x)), \textrm{ and } v \in V(\Gamma(y))\} \,.
$$

Note that in the definition of a bad edge, it is possible that $x=u$ or $y=v$.

Finally, we say that a vertex $x \in B$ is \emph{good} if
every edge in $G$ with exactly one endpoint in $V(\Gamma(x))$ (if
such an edge exists) has its other endpoint in $T$.

\begin{observation}\label{obs::equivalent}
Let $G$ be a graph with the partition $V(G) = T \cup M \cup B$ as
described above. The following conditions are equivalent:
\begin{enumerate}
\item $G$ satisfies properties (P1) and (P2).
\item $BAD_G \cap E(G) = \emptyset$.
\item Every vertex in $B$ is good.
\end{enumerate}
\end{observation}

We say that the graph $G$ is \emph{good} if it satisfies
conditions 1--3 of Observation~\ref{obs::equivalent}. Given this
definition, we make another observation.

\begin{observation}\label{obs::notBadIsGood}
Let $G$ be a good graph and let $G' = G \cup \{e\}$ for some $e
\notin E(G)$. $G'$ is a good graph if and only if $e \notin BAD_G$.
\end{observation}

The last definition we need is that of an almost good graph. A graph
$G$ is said to be \emph{almost good} if $G$ is good, or there exists
an edge $e \in E(G)$ such that $G \setminus \{e\}$ is a good graph
(the edge $e$ is not necessarily unique).

We now state and prove several claims that will be very useful in
the remainder of the proof. In all of these claims we assume that at
all times the graphs in question are 3-colorable.

\begin{claim}\label{c::topClassification}
Consider a graph $G$ with the corresponding partition $V(G) = T \cup
M \cup B$, and let $x \neq v_0$ be some vertex. The following hold:
\begin{enumerate} [(a)]
\item If $\Gamma(x)$ is not 2-colorable, then $x\in T$.
\item If $G$ satisfies property (P1) and $\Gamma(x)$ is 2-colorable, then $x\not\in T$.
\end{enumerate}
\end{claim}

\begin{proof}
For (a), since $\Gamma(x)$ is not 2-colorable, in every proper
3-coloring $c$ of $G$ there exists a vertex $v \in V(\Gamma(x))$ such
that $c(v) = c(v_0)$. Since $V(\Gamma(x)) \setminus \{x\} \subseteq M$
and no vertex in $M$ can receive the same color as $v_0$, it follows
that $v=x$ and thus $x \in T$.

For (b), we show that there exists a proper 3-coloring of $G$ which
does not assign the same color to $x$ and $v_0$. Indeed, since by definition
every edge with at least one endpoint in $T$ must have its
other endpoint in $M$ and since $B$ is an
independent set by assumption, it follows that $T \cup B$ is an
independent set. Since $\Gamma(x)$ is 2-colorable, and so is every
middle-component, it follows that $G[\{x\} \cup M]$ is 2-colorable.
Let $c$ be some proper coloring of $G[\{x\} \cup M]$ with colors 1
and 2. Extend $c$ to a coloring of $G$ by coloring each vertex in
$T\cup B \setminus \{x\}$ with the color 3. This is a proper
3-coloring of $G$ which assigns $x$ and $v_0$ distinct colors. We
conclude that $x \notin T$.
\end{proof}

\begin{claim} \label{c::BadEdgeKeepsB}
If $G$ is a good graph, $e \in BAD_G$ and $G' = G\cup \{e\}$, then
$B_{G'} = B_G$.
\end{claim}

\begin{proof}
As $G$ is a good graph, it follows by the contrapositive of
Claim~\ref{c::topClassification}(a) and by the fact that every
middle-component is always 2-colorable, that $G[M_G \cup B_G]$ is
also 2-colorable. Since, by the definition of a bad edge, $e$ connects
two different components of $G[M_G \cup B_G]$, it follows that
$G'[M_G \cup B_G]$ is 2-colorable as well. Thus, there exists a
proper coloring of $G'$ which assigns all vertices in $T_G$ the
color 1, and all the vertices in $M_G \cup B_G$ the colors 2 and 3.
It follows that $T_{G'} \subseteq T_G$, and therefore $T_{G'} = T_G$
since vertices are never moved from the top.

Now, since $T_{G'} = T_G$, and since $e \cap T_G = \emptyset$, it
follows by definition of the middle that $M_{G'} = M_G$, and
therefore we conclude that $B_{G'} = B_G$.
\end{proof}

\begin{claim}\label{c::deg2ToComponent}
Let $G$ be a good graph and let $G' = G \cup \{e\}$ for some $e
\notin E(G)$. If there exists a vertex $x$ such that $x \in B_G$ and
$x \in T_{G'}$, then there exists a middle-component $C \subseteq
M_{G'}$ such that $d_{G'}(x, C) \geq 2$.
\end{claim}

\begin{proof}
Since the addition of $e$ to $G$ moves $x$ from the bottom to the top,
Claim~\ref{c::BadEdgeKeepsB} implies that $e \notin BAD_G$.
Therefore, by Observation~\ref{obs::notBadIsGood}, $G'$ is a good
graph. By Claim~\ref{c::topClassification}(b),
$\Gamma_{G'}(x)$ is not 2-colorable, and thus contains an odd cycle.
Since $\Gamma_{G'}(x) \setminus \{x\} \subseteq M_{G'}$ is
2-colorable, this cycle must include $x$. The two neighbors of $x$
in the cycle belong to the same middle-component $C \subseteq
M_{G'}$, as claimed.
\end{proof}

\begin{claim}\label{c::gammaChange}
Let $G$ be an almost good graph with the partition $V(G) = T_G \cup
M_G \cup B_G$, and let $G_1 := G \cup \{ab\}$ for some $ab \notin
E(G)$ such that $a \in T_G$. Assume that there exists a vertex $x
\neq a$ such that $\Gamma_{G_1}(x) \neq \Gamma_G(x)$. Then $b \in
B_G$, $b \neq x$, and there exists an edge $ub \in E(G)$ such that
$u \in V(\Gamma_G(x))$.
\end{claim}

\begin{proof}
Since $G \subseteq G_1$ and no vertex can move out of the middle, it
follows that $M_G \subseteq M_{G_1}$. Therefore $\Gamma_G(x)
\subseteq \Gamma_{G_1}(x)$ and so if $\Gamma_{G_1}(x) \neq
\Gamma_G(x)$, then there must exist an edge $uw \in E(G_1)$ such
that $u \in V(\Gamma_G(x))$ and $w \in M_{G_1} \setminus
V(\Gamma_G(x))$. Recall that $a \neq x$ by assumption. Moreover, $a
\neq u$ as $a \notin M_G$. Since no vertex can move out of the top,
it follows that $a \in T_{G_1}$ and thus $a \neq w$. Therefore, $uw
\in E(G)$. However, $w \notin V(\Gamma_G(x))$ and this can only
happen if $w \notin M_G$. Since $w \in M_{G_1}$ we conclude that $w
\in B_G$. We will now show that $b=w$, which will complete the
proof, as $b \notin V(\Gamma_G(x))$ implies $b \neq x$.

Suppose for a contradiction that $b \neq w$. Since $w \in M_{G_1}$
there exists a vertex $z \in T_{G_1}$ such that $wz \in E(G_1)$. By
our assumption that $b \neq w$, and since $a \neq w$ as previously shown,
it follows that $wz \in E(G)$. Since $w \notin M_{G}$ it follows that $z
\notin T_G$ and therefore $z \in B_G$. Clearly the graph $G$ is not
good as the edge $wz$ violates property~(P1). However, since $G$ is
almost good, there exists an edge $e \in E(G)$ such that $G_0 := G
\setminus \{e\}$ is a good graph. Note that $w, z \in B_{G_0}$, and
therefore $e = wz$ as otherwise property~(P1) would have been violated
in $G_0$ as well.

Now, similarly to the proof of Claim~\ref{c::BadEdgeKeepsB}, there
exists a proper coloring $c$ of $G$ which assigns all vertices in
$T_G$ the color 1 and all the vertices in $M_G \cup B_G$ the colors
2 and 3. This coloring is also a proper coloring of $G_1$, since
$c(a) = 1$ and $c(b) \neq 1$. Since $z \in B_G$, it follows that
$c(z) \neq c(v_0)$ contrary to $z$ being an element of $T_{G_1}$.
We conclude that indeed $b=w$ and the proof is complete.
\end{proof}

The following result is an immediate consequence of Claim~\ref{c::gammaChange}.

\begin{corollary} \label{cor::NewGamma}
Under the assumptions of Claim~\ref{c::gammaChange}, and using the
same terminology, $\Gamma_{G_1}(x)$ is the subgraph of $G$ induced
on the vertex set $V(\Gamma_{G}(x)) \cup V(\Gamma_{G}(b))$.
\end{corollary}

\begin{claim}\label{c::xNotInTop}
Consider a good graph $G$ with the corresponding partition $V(G) =
T_G \cup M_G \cup B_G$. Let $G_1 = G \cup \{e\}$ for some $e \notin
E(G)$, let $x,y \in B_{G_1}$ and let $G_2 = G_1 \cup \{v_0y\}$.
Assume that $G_2$ satisfies property (P1). Then, $x \notin T_{G_2}$.
\end{claim}

\begin{proof}
Assume first that $\Gamma_{G_2}(x) = \Gamma_{G_1}(x)$. By the
assumption that $x \in B_{G_1}$ and by
Claim~\ref{c::topClassification}(a) we deduce that $\Gamma_{G_1}(x)$ is
2-colorable and thus so is $\Gamma_{G_2}(x)$. Since, moreover, $G_2$
satisfies property (P1) by assumption, it follows by
Claim~\ref{c::topClassification}(b) that $x \notin T_{G_2}$.

Assume then that $\Gamma_{G_2}(x) \neq \Gamma_{G_1}(x)$. As $G_1$ is
an almost good graph, by Claim~\ref{c::gammaChange} there is an edge
of $G_1$ between $V(\Gamma_{G_1}(x))$ and $y$. Since $G$ is a good
graph, it contains no edges between $V(\Gamma_G(x))$ and $y$.
Therefore, the edge $e$ must have one endpoint in $V(\Gamma_{G}(x))$
and one endpoint in $V(\Gamma_{G}(y))$. Let $C$ denote the connected
graph $\Gamma_{G}(x) \cup \Gamma_{G}(y) \cup \{e\}$. Since $x,y \in
B_G$, it follows by Claim~\ref{c::topClassification}(a) that
$\Gamma_{G}(x)$ and $\Gamma_{G}(y)$ are 2-colorable, and therefore
so is $C$. Since $\Gamma_{G_1}(x)$ and $\Gamma_{G_1}(y)$ are clearly
subgraphs of $C$, it follows by Corollary~\ref{cor::NewGamma} that
$\Gamma_{G_2}(x) = C$. Since $G_2$ satisfies property (P1), and
since $\Gamma_{G_2}(x)$ is 2-colorable, it follows by
Claim~\ref{c::topClassification}(b) that $x \not\in T_{G_2}$.
\end{proof}

Now, we present a strategy for Mini; it is divided into two simple
stages. In the first stage Mini claims only edges incident with
$v_0$, aiming to make its degree as large as possible, and in the
second stage she plays arbitrarily. For convenience we assume that
Max is the first player; if Mini is the first player, then in her
first move she claims $v_0 z$ for an arbitrary vertex $z \in B$ and
the remainder of the proof is essentially the same.
\begin{description}
\item [Stage I:] This stage lasts as long as there are vertices in
$B$. Once $B = \emptyset$, this stage is over and Mini proceeds to
Stage II. Before each of Mini's moves during Stage I, let $G$ denote
the graph at that point and let $e$ denote the
last edge claimed by Max. Mini plays as follows:
\begin{description}
\item [(i)] If there exists a vertex $x \in B_G$ such that $\{e\}
\cap \Gamma_G(x) \neq \emptyset$, then Mini claims $v_0x$ (if there
are several such bottom vertices, then Mini picks one arbitrarily).
\item [(ii)] Otherwise, Mini claims $v_0 z$, where $z \in B_{G}$ is an arbitrary vertex.
\end{description}
Mini then repeats Stage I.

\item [Stage II:] Throughout this stage, Mini follows the trivial strategy.
\end{description}

It remains to prove that Mini can indeed follow the proposed
strategy and that, by doing so, she ensures that $e(G) \leq 21
n^2/64 + O(n)$ will hold at the end of the game. Starting with the
former, note that Mini can clearly follow Stage II of the strategy.
As for Stage I, in each of her moves in this stage Mini claims an
edge between $v_0$ and some vertex $u \in B_{G}$. By definition
$v_0u$ is free and $\chi(G \cup \{v_0 u\}) \leq 3$. Hence Mini can
follow the proposed strategy. In order to prove the latter, we first
prove the following four additional claims.

\begin{claim}\label{c::MiniEnsuresGoodGraph}
Immediately after each of Mini's moves in Stage I, the current graph $G$ built by both players is good.
\end{claim}

\begin{proof}
We will prove this claim by induction on the number of moves played
by Mini. The claim clearly holds before the game starts. Assume that
it holds immediately after Mini's $i$th move for some non-negative
integer $i$ (where $i=0$ we refer to the initial graph before the
game begins); we will prove it holds after her $(i+1)$st move as
well (assuming it is played in Stage I). Let $G$ denote the graph
immediately after Mini's $i$th move, let $uv$ denote the edge
claimed by Max in his subsequent move, let $G_1 = G \cup \{uv\}$,
and let $G_2$ denote the graph immediately after Mini's $(i+1)$st
move.

If $uv \notin BAD_G$, then by Observation~\ref{obs::notBadIsGood}
$G_1$ is good. Mini then claims an edge $e$ with one
endpoint in $T_{G_1}$ (the vertex $v_0$), and so $e \notin BAD_{G_1}$
by definition. Therefore, applying Observation~\ref{obs::notBadIsGood}
once again we infer that $G_2$ is good.

Assume then that $uv$ is a bad edge. Therefore, by definition, there
exist distinct vertices $x,y \in B_G$ such that $u \in
V(\Gamma_G(x))$ and $v \in V(\Gamma_G(y))$. Note that according to
her strategy, in her next move Mini claims either $v_0 x$ or $v_0 y$
(by the induction hypothesis, no other bottom vertex is a
candidate); without loss of generality assume that she claims $v_0
y$. In order to prove that $G_2$ is good, we will show that every
vertex of $B_{G_2}$ is good. Consider first a vertex $z \in B_{G_2}
\setminus \{x\}$ (note that $z \neq y$ as $y \in M_{G_2}$). Clearly
$z \in B_G$ and since $G$ is a good graph, $z$ is a good vertex in
$G$. Since $\{uv, v_0y\} \cap \Gamma_G(z) = \emptyset$, it is easy
to see that $\Gamma_{G_2}(z) = \Gamma_{G_1}(z) = \Gamma_G(z)$ and
that $z$ is a good vertex in $G_2$ as well. Now consider $x$. Since
$G_1$ is an almost good graph, $\Gamma_{G_2}(x) = \Gamma_{G}(x) \cup
\Gamma_{G}(y) \cup \{uv\}$ by Corollary~\ref{cor::NewGamma}. Since
$x$ and $y$ are both good vertices in $G$, it is evident that $x$ is
a good vertex in $G_2$ as well. This concludes the proof of the
claim.
\end{proof}

\begin{claim}\label{c::MiniDoesntMoveToTop}
Throughout Stage I, no vertex is moved from $B$ to $T$ as a result of
a move by Mini.
\end{claim}

\begin{proof}
Recall that by assumption Max is the first player to move. Let $i$
be some positive integer, let $G$ denote the graph immediately
before Max's $i$th move, let $G_1$ denote the graph immediately
after Max's $i$th move, and let $G_2$ denote the graph immediately
after Mini's $i$th move. Since, by
Claim~\ref{c::MiniEnsuresGoodGraph}, $G$ and $G_2$ are good graphs,
and since Mini in her $i$th move claims $v_0y$ for some $y \in
B_{G_1}$, it follows by Claim~\ref{c::xNotInTop} that for every
vertex $x \in B_{G_1}$ (including $y$), $x \notin T_{G_2}$.
\end{proof}

\begin{claim}\label{c::deadRemains Dead}
Let $x$ be a bottom vertex which is attached to a middle-component
$C$. If at some point during Stage I $x$ is moved to the top, then
from this point until the end of Stage I, immediately after every
move of Mini, no bottom vertex will be attached to the unique
middle-component containing $C$.
\end{claim}

\begin{proof}
We prove this claim by induction on the number of rounds played
after $x$ was moved to the top. Consider first the moment at which
$x$ is moved to the top. By Claim~\ref{c::MiniDoesntMoveToTop} this
happens as a result of Max's $i_0$th move, for some positive integer
$i_0$. Denote the players' graph immediately before this move by
$G_0$ and the graph immediately after this move by $G_0'$. Since by
Claim~\ref{c::MiniEnsuresGoodGraph}, $G_0$ is a good graph, it is
not hard to see (similarly to the proof of
Claim~\ref{c::deg2ToComponent}) that in his $i_0$th move Max claimed
an edge $e \subseteq \Gamma_{G_0}(x)$, and thus $V(\Gamma_{G_0'}(x))
= V(\Gamma_{G_0}(x))$. Therefore, there are no edges of ${G_0'}$
between $V(\Gamma_{G_0'}(x))$ and $B_{G_0'}$, as $x$ itself is in
$T_{G_0'}$ by assumption, and it was the only vertex attached to the
middle-components of $\Gamma_{G_0}(x)$ (that is, the
middle-components of $G_0$ contained in $\Gamma_{G_0}(x)$) since
${G_0}$ is a good graph. In her subsequent move, Mini certainly does
not attach any vertex to any of the middle-components of
$\Gamma_{G_0}(x)$, nor does she change $\Gamma(x)$, so the claim
holds at this point.

Now let $i \geq i_0$ and assume the claim holds immediately after
Mini's $i$th move. Let $G_1$ be the graph after Mini's $i$th move,
let $G_2$ be the graph after Max's subsequent move, and let $G_3$ be
the graph after Mini's $(i+1)$st move. Let $C$ be a middle-component
of $\Gamma_{G_0}(x)$ and for $j = 1,2,3$ let ${C_j}$ denote the
middle-component containing $C$ in $G_j$. If there is no bottom
vertex attached to ${C_2}$ in $G_2$, then by Mini's strategy ${C_3}
= {C_2}$ and there will be no such vertex in $G_3$ either. Assume
then that there is such a vertex $y$. It follows that in his
$(i+1)$st move Max claimed an edge $uv$ such that $u \in {C_1}$ and
$v \in V(\Gamma_{G_1}(y))$. Therefore ${C_1} \subseteq
\Gamma_{G_2}(y)$. Since, by the induction hypothesis, there is no
vertex attached to ${C_1}$ in $G_1$ and since $y$ is the only vertex
attached to any middle-component of $\Gamma_{G_1}(y)$ in $G_1$ (by
property~(P2), as $G_1$ is a good graph), it follows that there is
no bottom vertex $z \neq y$ in $G_2$ such that $\{uv\} \cap
\Gamma_{G_2}(z) \neq \emptyset$. Therefore, by the proposed strategy
Mini claims $v_0y$ in her $(i+1)$st move and thus ${C_3} = {C_1}
\cup \Gamma_{G_1}(y) \cup \{uv\}$. It follows that no bottom vertex
is attached to ${C_3}$ in $G_3$.
\end{proof}

\begin{claim} \label{c::SizeOfTop}
$|T| \leq \frac{n+3}{4}$ holds at the end of Stage I.
\end{claim}

\begin{proof}
Consider the moment at which some vertex $x$ is moved from the
bottom to the top (if this never happens, then $|T| = 1$). At this
moment we \emph{assign} to $x$ every edge of $G$ which is incident
with $x$ and every edge of every middle-component to which $x$ is
attached. We claim that any edge of $G$ is assigned to at most one
vertex. Indeed, this is evident for the edges incident to the vertex
that was moved to the top, and is also true for the edges inside
the middle-components it was attached to by
Claims~\ref{c::MiniDoesntMoveToTop} and~\ref{c::deadRemains Dead}.
In addition, it follows by Claims~\ref{c::MiniDoesntMoveToTop}
and~\ref{c::deg2ToComponent} that every top vertex, other than
$v_0$, is assigned at least 3 edges. Since throughout Stage I Mini
claims only edges which are incident with $v_0$, all assigned edges
were claimed by Max. It thus follows that for every vertex of $T
\setminus \{v_0\}$, Mini increased the degree of $v_0$ by at least
3, that is, $|M| \geq d(v_0) \geq 3 (|T|-1)$. Since $B = \emptyset$ holds by
definition at the end of Stage I, it follows that $|T| + |M| = n$
holds at that point. We conclude that $|T| \leq \frac{n+3}{4}$ as
claimed.
\end{proof}

\vspace*{0.5cm}

We can now complete the proof of Theorem~\ref{th::3colorability}.
Let $X, Y$ and $Z$ denote the color classes in the unique proper
$3$-coloring of $G$ at the end of the game. It follows by the
definition of $T, M$ and $B$ that without loss of generality $X
\cup Y = M$ and $Z = T$. We thus conclude that
$$
e(G) = |T| (|X| + |Y|) + |X||Y| \leq \frac{n+3}{4} \cdot
\frac{3n-3}{4} + \left(\frac{3n-3}{8} \right)^2 = \frac{21}{64} n^2
+ O(n)
$$
as claimed. {\hfill $\Box$ \medskip\\}

\textbf{Proof of Theorem~\ref{th::Kcolorability}} For convenience we
assume that Mini is the first player; if Max is the first player,
then he makes an arbitrary first move and the remainder of the proof
is essentially the same. Let $k$, $C$ and $n$ be as in the statement
of the theorem; by choosing $C$ to be sufficiently large, we can
assume that $k$ is large as well, as otherwise the statement of the theorem trivially holds. Since the game in question is a
finite, perfect information game, with no chance moves, then exactly
one of the following must hold:
\begin{enumerate} [(a)]
\item Max has a strategy to ensure that $e(G) \geq \left(1 - C \log k/k \right)
\binom{n}{2}$ will hold at the end of the game against any strategy
of Mini.
\item Mini has a strategy to ensure that $e(G) < \left(1 - C \log k/k \right)
\binom{n}{2}$ will hold at the end of the game against any strategy
of Max.
\end{enumerate}

We present a random strategy for Max, and show that with positive
probability (in fact, a.a.s.), $e(G) \geq \left(1 - C
\log k/k \right) \binom{n}{2}$ will hold at the end of the game.
Therefore, (b) cannot hold, which implies (a), thus proving
Theorem~\ref{th::Kcolorability}.

We will need the following Chernoff type bound for our calculations.
\begin{theorem} \cite{JLR} \label{th::Chernoff}
Let $X \sim Bin(n,p)$ and let $x \geq 7 n p$. Then $Pr(X \geq x)
\leq e^{-x}$.
\end{theorem}

The proposed strategy for Max is divided into the following two
stages:
\begin{description}
\item [Stage I:] This stage is over as soon as $\delta(G) \geq k-1$; at that point Max proceeds to Stage II. For every positive integer $i$, let $a_i b_i$ denote the edge claimed by Mini in her $i$th move of this stage and, immediately after Mini's $i$th move, let $S_i = \{x \in \{a_i, b_i\} : d_G(x) \leq k-2\}$. Max plays his $i$th move as follows:
\begin{description}
\item [(i)] If $S_i = \emptyset$, then Max claims a free edge $xy$ such that $\min \{d_G(x), d_G(y)\} \leq k-2$ uniformly at random among all such edges; we refer to such moves as being \emph{fully-random}.

\item [(ii)] If $S_i \neq \emptyset$, then, with probability $139/140$, Max makes a fully-random move and, with probability $1/140$, he claims an edge $xy$ such that $x \in S_i$ and subsequently $y \in \{z \in V(K_n) : xz \notin E(G)\}$ are chosen uniformly at random; we refer to such moves as being \emph{semi-random}.
\end{description}
\item [Stage II:] Throughout this stage, Max follows the trivial strategy.
\end{description}

Note that if $H$ is a graph with chromatic number $\chi(H) \leq k$
and $u, v \in V(H)$ are two vertices such that $d_H(u) \leq k-2$,
then $\chi(H + uv) \leq k$. It thus follows that Max can play according
to the proposed strategy.

Fix $r := C n \log k/k$ and let $U \subseteq V(K_n)$ be an arbitrary
set of size $r$. Assume that for some positive integer $i$, in her
$i$th move Mini claims an edge $a_i b_i$ such that $S_i \cap U =
\emptyset$. Max's $i$th move $xy$ is said to be \emph{bad with
respect to $U$} if it is semi-random and $\{x,y\} \cap U \neq
\emptyset$. The set $U$ is said to be \emph{bad} if throughout Stage
I Max makes at least $kr/20$ bad moves with respect to $U$. We claim
that a.a.s.\ there will be no bad sets. Indeed, fix some $U$ of size $r$
and let $B_U$ be the random variable which counts the number of bad
moves with respect to $U$. Throughout the game Mini can claim at
most $(k-2)n$ edges $a_i b_i$ for which $S_i \neq \emptyset$. Hence,
it follows by the proposed strategy that $B_U \sim Bin(N,p)$, where
$N \leq (k-2) n$ and $p \leq \frac{1}{140} \cdot \frac{r}{n - k +
1}$. In particular, $\mathbb{E}(B_U) = N p \leq k r/140$. Therefore,
by Theorem~\ref{th::Chernoff} we have
\begin{equation*}
Pr(U \textrm{ is bad}) = Pr(B_U \geq kr/20) \leq e^{- k r/20} \,.
\end{equation*}
Thus
\begin{equation*}
Pr(\textrm{there exists a bad set } U) \leq \binom{n}{r} e^{- k
r/20} \leq \left[\frac{e k}{C \log k} \cdot e^{- k/20} \right]^r =
o(1) \,.
\end{equation*}
Hence, from now on we will assume that there are no bad sets.

Our next goal is to prove that a.a.s.\ $\alpha(G) \leq C n \log k/k$ holds
at the end of the game. Let $U \subseteq
V(G)$ be an arbitrary set of size $r$. At any point during the game
let $X_U = \{x \in U : d_G(x) \geq k-1\}$ and let $Y_U = U \setminus
X_U$. Consider the point in time at which $|X_U| \geq |Y_U|$ first
occurs; clearly $|U|/2 \leq |X_U| \leq \lceil |U|/2 \rceil + 1$ at
this point. Note that such a moment must occur during Stage I, since
there are no vertices of degree at most $k-2$ at the end of Stage I;
denote this moment by $t$.

We say that a claimed edge $uv$ is a $Y$-edge if
$\{u, v\} \cap Y_U \neq \emptyset$ holds immediately after $uv$ is
claimed (by either player). Let $A_U$ denote the event: ``up until
the moment $t$, Max has played at least $k r/20$ fully-random moves
in which he claimed $Y$-edges" and let $A_U^c$ denote its
complement. Let $I_U$ denote the event: ``at the end of the game, $U$
is an independent set". Clearly $Pr(I_U) = Pr(I_U \wedge A_U) +
Pr(I_U \wedge A_U^c)$.

We wish to bound $Pr(I_U)$ from above. Consider first a fully-random
move $e$ such that $e$ is a $Y$-edge,
%; assume w.l.o.g.\ that $x \in Y_U$. Clearly we can
and assume that $U$ is independent immediately before this move. We
have
%\begin{eqnarray} \label{eq::fullyRandomMove}
%Pr(\{x,y\} \subseteq U) &\geq& \frac{1}{2} \sum_{v \in Y_U} Pr(x = v
%\; | \; x \in Y_U)
%\cdot \frac{d_{\overline{G}}(v, U)}{d_{\overline{G}}(v)} \nonumber\\
%&\geq& \frac{r-1}{2n} \sum_{v \in Y_U} Pr(x = v \; | \; x \in Y_U)
%\geq \frac{r}{3n} \,.
%\end{eqnarray}
\begin{eqnarray} \label{eq::fullyRandomMove}
Pr(e \subseteq U) &\geq& \frac{|Y_U||U| - \binom{|Y_U|}{2}}{|Y_U|n}  \\
&\geq& \frac{|Y_U|\frac{|U|}{2}}{|Y_U|n} = \frac{r}{2n} \,.
\end{eqnarray}

It follows by~\eqref{eq::fullyRandomMove} and by the definition of
$A_U$ that
\begin{equation} \label{eq::fullyRandom}
Pr(I_U \wedge A_U) = Pr(I_U \; | \; A_U) \cdot Pr(A_U) \leq Pr(I_U
\; | \; A_U) \leq \left(1 - \frac{r}{2n} \right)^{k r/20} \,.
\end{equation}

Next, assume that for some positive integer $i$, in her $i$th move
Mini has claimed a $Y$-edge $a_i b_i$; assume without loss of generality that $a_i \in Y_U$. We again assume
that $U$ is independent immediately after Mini's $i$th move; in
particular, $b_i \in V(K_n) \setminus U$. According to the proposed
strategy, with probability $1/140$ in his subsequent move Max
claims an edge $xy$ such that $x \in \{a_i, b_i\}$ and $y \in
V(K_n)$. Moreover, $Pr(x = a_i) \geq 1/2$ and, based on the assumption that $U$ is
independent, $Pr(y \in U \; | \; x = a_i) \geq \frac{|U
\setminus \{a_i\}|}{n} \geq \frac{r-1}{n}$. Therefore
\begin{equation} \label{eq::semiRandomMove}
Pr(\{x,y\} \subseteq U) \geq 1/140 \cdot Pr(x = a_i) \cdot Pr(y \in
U \; | \; x = a_i) \geq \frac{r}{290n} \,.
\end{equation}

Since $X_U = \emptyset$ before the game starts, if $U$ is an
independent set at the end of the game, it follows that by time $t$
the number of $Y$-edges claimed by both players is at least $(k-2)
|X_U| \geq (k-2) r/2$. Since $U$ is not a bad set, if by time $t$
Max makes at most $kr/20$ fully-random moves in which he claims
$Y$-edges, then the number of $Y$-edges Mini must claim up to that
point is at least $[(k-2) r/2 - kr/20 - kr/20]/2 \geq kr/10$. It
thus follows by the proposed strategy, by~\eqref{eq::semiRandomMove}
and by the definition of $A_U$, that
\begin{equation} \label{eq::semiRandom}
Pr(I_U \wedge A_U^c) = Pr(I_U \; | \; A_U^c) \cdot Pr(A_U^c) \leq
Pr(I_U \; | \; A_U^c) \leq \left(1 - \frac{r}{290n} \right)^{k r/10}
\,.
\end{equation}

Putting inequalities~\eqref{eq::fullyRandom}
and~\eqref{eq::semiRandom} together we conclude that
\begin{eqnarray} \label{eq::oneSet}
Pr(I_U) &=& Pr(I_U \wedge A_U) + Pr(I_U \wedge A_U^c) \leq \left(1 - \frac{r}{2n} \right)^{k r/20} + \left(1 - \frac{r}{290n} \right)^{k r/10} \nonumber \\
&\leq& \exp \left\{- \frac{r}{2n} \cdot \frac{k r}{20} \right\} +
\exp \left\{- \frac{r}{290n} \cdot \frac{k r}{10} \right\} \leq \exp
\left\{- \frac{C^2 n \log^2 k}{3000 k} \right\}\,.
\end{eqnarray}

Using~\eqref{eq::oneSet}, we can now show that
a.a.s.\ $\alpha(G) \leq C n \log k/k$ holds at the end of the game by
the following union bound estimate:
\begin{eqnarray*} \label{eq::independenceNumber}
Pr(\alpha(G) \geq C n \log k/k) &\leq& \binom{n}{C n \log k/k} \cdot \exp \left\{- \frac{C^2 n \log^2 k}{3000 k} \right\}\\
&\leq& \left(\frac{e k}{C \log k} \right)^{C n \log k/k} \cdot \exp \left\{- \frac{C^2 n \log^2 k}{3000 k} \right\}\\
&\leq& \exp \left\{\frac{C n \log^2 k}{k} - \frac{C^2 n \log^2 k}{3000 k} \right\}\\
&=& o(1) \,,
\end{eqnarray*}
where the last equality holds for $C > 3000$. Since $\alpha(G) \leq
C n \log k/k$ holds a.a.s., Max has a deterministic strategy to
achieve this, as was shown above.

Once the game is over, $G$ is
saturated and thus complete $k$-partite; let $A_1, \ldots, A_k$
denote its parts. Max can ensure that $|A_i| \leq C n \log k/k$
holds for every $1 \leq i \leq k$, and therefore
$$
e(G) = \frac{1}{2} \sum_{i=1}^k |A_i| (n - |A_i|) \geq \frac{1}{2}
\sum_{i=1}^k |A_i| (n - C n \log k/k) = \frac{n^2}{2} (1 - C \log
k/k) \,.
$$
{\hfill $\Box$ \medskip\\}

\section{Matching games} \label{sec::matching}

In this section we study matching games, that is, saturation games
in which both players are required to keep the size of every
matching in the graph below a certain threshold.

\textbf{Proof of Theorem~\ref{th::perfectMatchingGame}} In order to
prove the theorem, we present a strategy for Max. In order to
simplify the description of the strategy, we first consider several
possible \emph{end-games}. These are described in the following
lemmas.

\begin{lemma} \label{lem::nMinus2cycle}
Let $n \geq 6$ be an even integer and let $G_0 = (V,E)$ be a graph
on $n$ vertices. Assume that there exist vertices $x,y \in V$ such
that $d_{G_0}(x) = d_{G_0}(y) = 0$ and $G_0 \setminus \{x,y\}$
admits a Hamilton cycle $C$. If Max is the second player, then
$s(G_0, \mathcal{PM}) \geq \binom{n-2}{2}$.
\end{lemma}

\begin{proof}
Max plays according to the following simple strategy which consists
of two stages.
\begin{description}
\item [Stage I:] Let $uv$ denote the last edge claimed by Mini; we distinguish between the following two cases:
\begin{description}
\item [(1)] If $\{u,v\} \cap \{x,y\} = \emptyset$, then Max claims an arbitrary free edge $w w'$ such that $\{w,w'\} \cap \{x,y\} = \emptyset$ and repeats Stage I; if this is not possible, then he skips to Stage II.
\item [(2)] Otherwise, if $u \in \{x,y\}$ and $v \in V \setminus \{x,y\}$, then Max claims a free edge $u v'$, where $v'$ is a neighbor of $v$ in $C$. He then proceeds to Stage II.
\end{description}
\item [Stage II:] Throughout this stage, Max follows the trivial strategy.
\end{description}

Note that at any point during the game, the graph $G \cup \{xy\}$
admits a perfect matching; it follows that $xy \notin E(G)$. In
particular, the proposed strategy does account for every legal move
of Mini. Moreover, if Max never plays according to Case (2) of Stage
I, then clearly $w w' \in E(G)$ holds for every $w, w' \in V
\setminus \{x,y\}$ at the end of the game. If on the other hand Max
does play according to Case (2) of Stage I, then, at the end of the
game, $w w' \in E(G)$ holds for every $w, w' \in V \setminus \{z\}$
for some $z \in \{x,y\}$. In either case we conclude that $s(G_0,
\mathcal{PM}) \geq \binom{n-2}{2}$ as claimed.
\end{proof}

\begin{lemma} \label{lem::outsideEdge}
Let $n \geq 6$ be an even integer and let $G_0 = (V,E)$ be a graph
on $n$ vertices. Assume that there exist vertices $x,y,z \in V$ such
that $xy \in E$, $d_{G_0}(x) = d_{G_0}(y) = 1$, $d_{G_0}(z) = 0$ and
$G_0 \setminus \{x,y,z\}$ admits a Hamilton cycle $C$. If Max is the
second player, then $s(G_0, \mathcal{PM}) \geq \binom{n-3}{2}$.
\end{lemma}

\begin{proof}
Max plays according to the following simple strategy which consists
of two stages.
\begin{description}
\item [Stage I:] Let $uv$ denote the last edge claimed by Mini; we distinguish between the following three cases:
\begin{description}
\item [(1)] If $\{u,v\} \cap \{x,y,z\} = \emptyset$, then Max claims an arbitrary free edge $w w'$ such that $\{w,w'\} \cap \{x,y,z\} = \emptyset$ and repeats Stage I; if this is not possible, then he skips to Stage II.
\item [(2)] Otherwise, if $uv = xz$ (respectively $uv = yz$), then Max claims $yz$ (respectively $xz$) and proceeds to Stage II.
\item [(3)] Otherwise, if $u \in \{x, y\}$ and $v \in V \setminus \{x,y,z\}$, then Max claims a free edge $u' v'$, where $u'$ is the unique vertex in $\{x, y\} \setminus \{u\}$ and $v'$ is a neighbor of $v$ in $C$. He then proceeds to Stage II.
\end{description}
\item [Stage II:] Throughout this stage, Max follows the trivial strategy.
\end{description}

Note that at any point during the game, for every $w \in V \setminus
\{x,y,z\}$, the graph $G \cup \{wz\}$ admits a perfect matching; it
follows that $wz \notin E(G)$. In particular, the proposed strategy
does account for every legal move of Mini. Moreover, note that if
$\{xz, yz\} \subseteq E(G)$, then $w w' \notin E(G)$ for every $w
\in \{x,y,z\}$ and $w' \in V \setminus \{x,y,z\}$. Therefore, if Max
never plays according to Case (3) of Stage I, then $w w' \in E(G)$
holds for every $w, w' \in V \setminus \{x,y,z\}$ at the end of the
game. If on the other hand Max does play according to Case (3) of
Stage I, then $w w' \in E(G)$ holds for every $w, w' \in V \setminus
\{z\}$ at the end of the game. In either case we conclude that
$s(G_0, \mathcal{PM}) \geq \binom{n-3}{2}$ as claimed.
\end{proof}

\begin{lemma} \label{lem::hangingEdge}
Let $n \geq 6$ be an even integer and let $G_0 = (V,E)$ be a graph
on $n$ vertices. Assume that there exist vertices $w,x,y,z \in V$
such that $wx \in E$, $d_{G_0}(x) = 1$, $d_{G_0}(y) = d_{G_0}(z) =
0$ and $G_0 \setminus \{x,y,z\}$ admits a Hamilton cycle $C$. If Max
is the second player, then $s(G_0, \mathcal{PM}) \geq
\binom{n-2}{2}$.
\end{lemma}

\begin{proof}
Max plays according to the following simple strategy which consists
of two stages.
\begin{description}
\item [Stage I:] Let $uv$ denote the last edge claimed by Mini; we distinguish between the following three cases:
\begin{description}
\item [(1)] If $\{u,v\} \cap \{y,z\} = \emptyset$, then Max claims an arbitrary free edge $a b$ such that $\{a, b\} \cap \{y,z\} = \emptyset$ and repeats Stage I; if this is not possible, then he skips to Stage II.
\item [(2)] Otherwise, if $u \in \{y,z\}$ and $v \in V \setminus \{x,y,z\}$, then Max claims $ux$ and proceeds to Stage II.
\item [(3)] Otherwise, if $u \in \{y,z\}$ and $v = x$, then Max claims an arbitrary free edge $u u'$, where $u' \in V \setminus \{x,y,z\}$. He then proceeds to Stage II.
\end{description}
\item [Stage II:] Throughout this stage, Max follows the trivial strategy.
\end{description}

Note that at any point during the game, the graph $G \cup \{yz\}$
admits a perfect matching; it follows that $yz \notin E(G)$. In
particular, the proposed strategy does account for every legal move
of Mini. Moreover, if Max never plays according to Cases (2) and (3)
of Stage I, then clearly $w w' \in E(G)$ holds for every $w, w' \in
V \setminus \{y,z\}$ at the end of the game. If on the other hand
Max does play according to Cases (2) or (3) of Stage I, then without loss of generality $xy \in E(G)$ (otherwise $xz \in E(G)$ and the
proof can be completed by an analogous argument). In these cases,
Max claims an edge and immediately proceeds to Stage II. Note that
starting from that point and until the end of the game, $G \setminus
\{z,t\}$ admits a perfect matching for every $t \in V$. Hence
$d_G(z) = 0$, and it follows that $w w' \in E(G)$ holds for every
$w, w' \in V \setminus \{z\}$ at the end of the game. In either case
we conclude that $s(G_0, \mathcal{PM}) \geq \binom{n-2}{2}$ as
claimed.
\end{proof}

\begin{lemma} \label{lem::triangleVertex}
Let $n \geq 8$ be an even integer and let $G_0 = (V,E)$ be a graph
on $n$ vertices. Assume that there exist vertices $w_1, w_2, w_3,
w_4 \in V$ such that $G_0[\{w_1, w_2, w_3\}] \cong K_3$,
$d_{G_0}(w_1) = d_{G_0}(w_2) = d_{G_0}(w_3) = 2$, $d_{G_0}(w_4) = 0$
and $G_0 \setminus \{w_1, w_2, w_3, w_4\}$ admits a Hamilton cycle
$C$. If Max is the second player, then $s(G_0, \mathcal{PM}) \geq
\binom{n-4}{2}$.
\end{lemma}

\begin{proof}
Max plays according to the following simple strategy which consists
of two stages.
\begin{description}
\item [Stage I:] Let $uv$ denote the last edge claimed by Mini; we distinguish between the following three cases:
\begin{description}
\item [(1)] If $\{u,v\} \cap \{w_1, w_2, w_3, w_4\} = \emptyset$, then Max claims an arbitrary free edge $xy$ such that $\{x, y\} \cap \{w_1, w_2, w_3, w_4\} = \emptyset$ and repeats Stage I; if this is not possible, then he proceeds to Stage II.
\item [(2)] Otherwise, if $u = w_4$ and $v \in V \setminus \{w_1, w_2, w_3, w_4\}$, then Max claims $u v'$, where $v'$ is a neighbor of $v$ in $C$. He then proceeds to Stage II.
\item [(3)] Otherwise, if $u \in \{w_1, w_2, w_3\}$ and $v \in V \setminus \{w_1, w_2, w_3, w_4\}$, then Max claims a free edge $u' v'$, where $u' \in \{w_1, w_2, w_3\} \setminus \{u\}$ and $v'$ is a neighbor of $v$ in $C$. He then proceeds to Stage II.
\end{description}
\item [Stage II:] Throughout this stage, Max follows the trivial strategy.
\end{description}

Note that at any point during the game, the graph $G \cup \{w_i
w_4\}$ admits a perfect matching for every $1 \leq i \leq 3$; it
follows that $w_i w_4 \notin E(G)$. In particular, the proposed
strategy does account for every legal move of Mini. Moreover, if Max
proceeds from Case (1) to Stage II, then clearly $xy \in E(G)$ holds
for every $x, y \in V \setminus \{w_1, w_2, w_3, w_4\}$ at the end
of the game. Similarly, if Max proceeds from Case (2) to Stage II,
then $xy \in E(G)$ holds for every $x, y \in V \setminus \{w_1, w_2,
w_3\}$ at the end of the game. Finally, if Max proceeds from Case
(3) to Stage II, then $xy \in E(G)$ holds for every $x, y \in V
\setminus \{w_4\}$ at the end of the game. In either case we
conclude that $s(G_0, \mathcal{PM}) \geq \binom{n-4}{2}$ as claimed.
\end{proof}

\begin{lemma} \label{lem::edge2vertices}
Let $n \geq 8$ be an even integer and let $G_0 = (V,E)$ be a graph
on $n$ vertices. Assume that there exist vertices $w_1, w_2, w_3,
w_4 \in V$ such that $w_3 w_4 \in E$, $d_{G_0}(w_1) = d_{G_0}(w_2) =
0$, $d_{G_0}(w_3) = d_{G_0}(w_4) = 1$ and $G_0 \setminus \{w_1, w_2,
w_3, w_4\}$ admits a Hamilton cycle $C$. If Max is the second player,
then $s(G_0, \mathcal{PM}) \geq \binom{n-4}{2}$.
\end{lemma}

\begin{proof}
Max plays according to the following simple strategy which consists
of two stages.
\begin{description}
\item [Stage I:] Let $uv$ denote the last edge claimed by Mini; we distinguish between the following four cases:
\begin{description}
\item [(1)] If $\{u,v\} \cap \{w_1, w_2, w_3, w_4\} = \emptyset$, then Max claims an arbitrary free edge $xy$ such that $\{x, y\} \cap \{w_1, w_2, w_3, w_4\} = \emptyset$ and repeats Stage I; if this is not possible, then he proceeds to Stage II.
\item [(2)] Otherwise, if $u \in \{w_1, w_2\}$ and $v \in V \setminus \{w_1, w_2, w_3, w_4\}$, then Max claims $u v'$, where $v'$ is a neighbor of $v$ in $C$. He then follows the strategy described in the proof of Lemma~\ref{lem::outsideEdge} until the end of the game.
\item [(3)] Otherwise, if $u \in \{w_3, w_4\}$ and $v \in V \setminus \{w_1, w_2, w_3, w_4\}$, then Max claims a free edge $u' v'$, where $v'$ is a neighbor of $v$ in $C$ and $u'$ is the unique vertex in $\{w_3, w_4\} \setminus \{u\}$. He then follows the strategy described in the proof of Lemma~\ref{lem::nMinus2cycle} until the end of the game.
\item [(4)] Otherwise, if $u \in \{w_1, w_2\}$ and $v \in \{w_3, w_4\}$, then Max claims $u v'$, where $v'$ is the unique vertex in $\{w_3, w_4\} \setminus \{v\}$. He then follows the strategy described in the proof of Lemma~\ref{lem::triangleVertex} until the end of the game.
\end{description}
\item [Stage II:] Throughout this stage, Max follows the trivial strategy.
\end{description}

Note that at any point during the game, the graph $G \cup \{w_1
w_2\}$ admits a perfect matching; it follows that $w_1 w_2 \notin
E(G)$. In particular, the proposed strategy does account for every
legal move of Mini. Moreover, if Max never plays according to Cases
(2), (3) and (4) of Stage I, then clearly $xy \in E(G)$ holds for
every $x, y \in V \setminus \{w_1, w_2, w_3, w_4\}$ at the end of
the game. If on the other hand Max does play according to Cases (2),
(3) or (4) of Stage I, then it follows by
Lemmas~\ref{lem::outsideEdge}, \ref{lem::nMinus2cycle}
and~\ref{lem::triangleVertex}, respectively, that $s(G_0,
\mathcal{PM}) \geq \binom{n-4}{2}$. In either case we conclude that
$s(G_0, \mathcal{PM}) \geq \binom{n-4}{2}$ as claimed.
\end{proof}

\bigskip

We can now describe Max's strategy for the perfect matching game
$(n, \mathcal{PM})$. At any point during Stages I -- III, if Max is
unable to follow the proposed strategy, then he skips to Stage IV.
The proposed strategy is divided into the following four stages.
\begin{description}
\item [Stage I:] Max follows the long path strategy until $G$ contains a path $P = (u_0, \ldots, u_{\ell})$ of length $\ell \in \{n-5, n-4\}$ which includes all vertices of positive degree. At that moment, if $\ell = n-4$, then Max skips to Stage III, otherwise he proceeds to Stage II.
\item [Stage II:] Let $V(G) \setminus V(P) = \{w_1, w_2, w_3, w_4\}$. Let $uv$ denote the edge Mini claims in her subsequent move; we distinguish between the following two cases:
\begin{description}
\item [(1)] If $\{u,v\} \cap V(P) \neq \emptyset$, then Max plays as follows. If $\{u,v\} \subseteq V(P)$, then Max claims $u_{\ell} w_4$. Otherwise, assume without loss of generality that $u \notin V(P)$. Max then claims $u_{\ell} u$ if it is free and $u_0 u$ otherwise.
In either case he extends $P$ to a path of length $n-4$. By abuse of
notation and for simplicity of presentation, we denote this path by $P = (u_0,
\ldots, u_{\ell})$ as well. Max then proceeds to Stage III.
\item [(2)] Otherwise, assume without loss of generality that $u = w_3$ and $v = w_4$. Max claims $u_0 u_{\ell}$, and then follows the strategy described in the proof of Lemma~\ref{lem::edge2vertices} until the end of the game.
\end{description}
\item [Stage III:] Let $V(G) \setminus V(P) = \{w_1, w_2, w_3\}$. Let $uv$ denote the edge Mini claims in her subsequent move; we distinguish between the following three cases:
\begin{description}
\item [(1)] If $\{u,v\} \subseteq V(P)$, then Max claims a free edge $x x'$ such that $\{x, x'\} \subseteq V(P)$ and repeats Stage III.
\item [(2)] Otherwise, if $\{u,v\} \subseteq \{w_1, w_2, w_3\}$, then Max claims $u_0 u_{\ell}$ if it is free and an arbitrary free edge $x x'$ such that $\{x, x'\} \subseteq V(P)$ otherwise. He then follows the strategy described in the proof of Lemma~\ref{lem::outsideEdge} until the end of the game.
\item [(3)] Otherwise, assume without loss of generality that $u \in V(P)$ and $v \in \{w_1, w_2, w_3\}$. Max claims $u_0 u_{\ell}$ if it is free and an arbitrary free edge $x x'$ such that $\{x, x'\} \subseteq V(P)$ otherwise. He then follows the strategy described in the proof of Lemma~\ref{lem::hangingEdge} until the end of the game.
\end{description}
\item [Stage IV:] Throughout this stage, Max follows the trivial strategy.
\end{description}

It remains to prove that Max can indeed follow the proposed strategy
and that, by doing so, he ensures that $e(G) \geq \binom{n-4}{2}$
holds at the end of the game. Starting with the former, note that
Max can follow Stage I of the proposed strategy by
Lemma~\ref{lem::longPath} (throughout Stage I there are isolated
vertices in $G$ and thus it does not admit a perfect matching). He
can follow Stage II by Lemma~\ref{lem::edge2vertices} and by the
properties of $P$, and can follow Stage III by
Lemmas~\ref{lem::outsideEdge} and~\ref{lem::hangingEdge}. Finally,
it is obvious that he can follow Stage IV of the proposed strategy.

As for the latter, if Max does not reach Stage IV of the proposed
strategy, then it follows by
Lemmas~\ref{lem::edge2vertices},~\ref{lem::outsideEdge}
and~\ref{lem::hangingEdge} that $e(G) \geq \binom{n-4}{2}$ holds at
the end of the game. If on the other hand Max does reach Stage IV of
the proposed strategy, then it follows by the description of the
proposed strategy that $xy \in E(G)$ holds at the end of the game
for every $x, y \in V(P)$ and thus  $e(G) \geq \binom{n-4}{2}$.
{\hfill $\Box$ \medskip\\}

\textbf{Proof of Theorem~\ref{th::kMatching}} Throughout this proof,
we assume that $n \geq 2k$, as otherwise $s(n, {\mathcal M}_k) =
\binom{n}{2}$ and so the assertion of the theorem holds trivially.
We will use the following terminology: the \emph{parity} of a player
is \emph{odd} if he is the first to move and \emph{even} otherwise.
Assume first that the parity of Max is opposite to the parity of
$k$. In order to prove that $s(n, {\mathcal M}_k) \geq n-1$, we
present a strategy for Max. Before doing so, we prove the following
auxiliary lemma.

\begin{lemma}\label{lem::connectAllIsolated}
Let $k \geq 2$ be an integer and let $G_0 = (V,E)$ be a graph.
Assume that there exists a partition $V = U \cup W$ such that
$\nu(G_0) = \nu(G_0 \setminus W) = k-1$. Assume further that there
exist vertices $w_1, w_2 \in W$ and $u \in U$ such that
$d_{G_0}(w_1) = d_{G_0}(w_2) = 1$ and $\{uw_1, uw_2\} \subseteq E$.
Then Max, as the second player, has a strategy to ensure that at the
end of the $(G_0, \mathcal{M}_k)$ game, $d_G(w, U) \geq 1$ will hold
for every $w \in W$.
\end{lemma}

\begin{proof}
We present a strategy for Max; it is divided into the following two
stages.

\begin{description}
\item [Stage I:] At any point during this stage, let $I = \{w \in W : d_G(w, U) = 0\}$. If $I = \emptyset$, then this stage is over and Max proceeds to Stage II. Otherwise, Max claims $uw$, where $w \in I$ is an arbitrary vertex.
\item [Stage II:] Throughout this stage, Max follows the trivial strategy.
\end{description}

It is evident that, if Max is able to follow the proposed strategy,
then $d_G(w, U) \geq 1$ holds for every $w \in W$ at the end of the
game. It thus suffices to prove that he can indeed do so. We will
prove this by induction on the size of $I$ in the beginning of the
game. If $|I| = 0$, then there is nothing to prove, as clearly Max
can follow Stage II of the proposed strategy. Assume that our claim
holds if $|I| \leq m$ for some non-negative integer $m$; we will
prove it holds for $m+1$ as well. Let $xy$ denote the edge Mini
claims in her first move. Since no edge with both endpoints in $W$
is legal (with respect to $G_0$ and $\mathcal{M}_k$), we can assume
without loss of generality that $x \in U$. In particular, $|\{x, y\}
\cap \{w_1, w_2\}| \leq 1$ and thus we can assume that $y \neq w_1$.
If $I = \emptyset$ holds immediately after this move, then there is
nothing to prove; hence, let $z \in I$ be an arbitrary vertex and
assume that Max claims $uz$ in his first move. Suppose for a
contradiction that this is not a legal move, that is, that $H := G_0
\cup \{xy, uz\}$ admits a matching of size $k$. Since this matching
must contain $uz$, and no matching of $H$ can cover both $z$ and
$w_1$, it follows that $\nu(H \setminus \{w_1\}) = k$. On the other
hand, $\nu(H \setminus \{z\}) = k-1$ holds by assumption. This is a
contradiction as clearly $H \setminus \{w_1\}$ is isomorphic to $H
\setminus \{z\}$. Immediately after Max's first move, $z \in W
\setminus I$ and thus $|I| \leq m$. Moreover, $d_H(w_1) = d_H(z) =
1$ and $\{uw_1, uz\} \subseteq E(H)$. By the induction hypothesis,
we conclude that Max can follow the proposed strategy until the end
of the game.
\end{proof}

We can now describe Max's strategy for the $k$-matching saturation
game $(n, \mathcal{M}_k)$. At any point during the game, if Max is
unable to follow the proposed strategy, then he forfeits the game.
The proposed strategy is divided into the following three stages.

\begin{description}
\item [Stage I:] Max follows the long path strategy until $G$ contains a path $P = (u_0, \ldots, u_{\ell})$ of length $\ell \in \{2k-4, 2k-3\}$ which includes all vertices of positive degree. At that moment, if $\ell = 2k-4$, then Max proceeds to Stage II, otherwise he skips to Stage III.
\item [Stage II:] Let $wv$ denote the edge Mini claims in her subsequent move; we distinguish between the following two cases:
\begin{description}
\item [(1)] If $\{w,v\} \cap V(P) \neq \emptyset$, then Max plays as follows.
If $\{w,v\} \subseteq V(P)$, then Max claims $u_{\ell} z$ for an
arbitrary vertex $z \in V(G) \setminus V(P)$. Otherwise, assume
without loss of generality that $w \notin V(P)$. Max then claims
$u_{\ell} w$ if it is free and $u_0 w$ otherwise. In either case he
extends $P$ to a path of length $2k-3$. By abuse of notation and for
simplicity of presentation, we denote this path by $P = (u_0, \ldots, u_{\ell})$ as
well. Max then proceeds to Stage III.
\item [(2)] Otherwise, assume without loss of generality that $d_G(u_0) = 1$ (recall property (c) in Lemma~\ref{lem::longPath}). Max claims $u_1 z$ for an arbitrary isolated vertex $z$, and then follows the strategy described in the proof of Lemma~\ref{lem::connectAllIsolated}, with $U = \{w, v, u_1, \ldots, u_\ell\}$, $u = u_1$ and $\{w_1, w_2\} = \{z, u_0\}$, until the end of the game.
\end{description}
\item [Stage III:] Let $wv$ denote the edge Mini claims in her subsequent move; we distinguish between the following two cases:
\begin{description}
\item [(1)] If $\{w,v\} \subseteq V(P)$, then Max claims a free edge $x y$ such that $\{x, y\} \subseteq V(P)$ and repeats Stage III.
\item [(2)] Otherwise, assume without loss of generality that $w \in V(P)$ and $v \notin V(P)$. Max claims $wz$ for some arbitrary isolated vertex $z$, and then follows the strategy described in the proof of Lemma~\ref{lem::connectAllIsolated}, with $U = V(P)$, $u = w$ and $\{w_1, w_2\} = \{z, v\}$, until the end of the game.
\end{description}
\end{description}

It remains to prove that Max can indeed follow the proposed strategy
and that, by doing so, he ensures that $e(G) \geq n-1$ holds at the
end of the game. Starting with the former, note that Max can follow
Stage I of the proposed strategy by Lemma~\ref{lem::longPath}
(throughout Stage I there are at most $2k-2$ vertices of positive
degree in $G$ and thus $\nu(G) < k$). An analogous argument shows
that he can follow Case (1) of Stage II. Max can make his first move
in Case (2) of Stage II, as $n \geq 2k$ and immediately after this
move, there are exactly $2k$ vertices of positive degree in $G$ but
no matching of $G$ covers both $z$ and $u_0$. Moreover, he can
follow the remainder of Case (2) of Stage II by
Lemma~\ref{lem::connectAllIsolated}. Next, consider Stage III. Mini
cannot claim an edge $wv$ such that $\{w,v\} \cap V(P) = \emptyset$
as no such edge is legal. Therefore, Cases (1) and (2) of Stage III
account for every legal move of Mini. Suppose for a contradiction
that at some point during the game Max forfeits the game while
attempting to follow Case (1) of Stage III. Since every free edge
with both endpoints in $V(P)$ is clearly legal, it follows that no
such edges remain. Therefore, the total number of edges played thus
far is $\binom{2k-2}{2} = (k-1)(2k-3)$ and it is Max's turn to play.
Since Max's parity is opposite to that of $k$, this is a
contradiction. Moreover, Max can make his first move in Case (2) of
Stage III, as immediately after this move, there are exactly $2k$
vertices of positive degree in $G$ but no matching of $G$ covers
both $z$ and $v$. Finally, he can follow the remainder of Case (2)
of Stage III by Lemma~\ref{lem::connectAllIsolated}.

In order to prove that $e(G) \geq n-1$ holds at the end of the game,
we examine the graph $G$ at the end of the game. If the game ends
when Max plays according to Case (2) of Stage III, then $G$ is
connected and thus $e(G) \geq n-1$. Otherwise, the game ends when
Max plays according to Case (2) of Stage II. Suppose for a
contradiction that $e(G) < n-1$ holds at the end of the game; in
particular, $G$ must be disconnected. It thus follows by the
description of the proposed strategy, that $G$ consists of exactly
two connected components, $C_1 \supseteq V(P)$ and $C_2 \supseteq
\{w,v\}$. Since $P$ admits a matching of size $k-2$ and $\nu(G) <
k$, it follows that $C_2$ is either a star or a triangle. Since
$e(G) \geq n-1$ holds in the latter case, we can assume that $C_2$
is a star. However, any edge $xy$, where $x$ is the center of the
star and $y \in C_1$, is still legal in this case, contrary to our
assumption that the game is over.

Next, assume that the parity of Mini is opposite to the parity of
$k$. Since the case $k = 2$ was considered in~\cite{PV}, we can
assume that $k \geq 3$. In order to prove the theorem, we present a
strategy for Mini. In order to simplify the description of the
strategy, we first consider several possible \emph{end-games} which
are described in the following lemmas. Since these lemmas and their
proofs are quite similar to those of Lemmas~\ref{lem::nMinus2cycle}
-- \ref{lem::edge2vertices}, we will omit some of the details.
Though this is not always necessary, in each of these lemmas we
assume that Mini is the second player.

\begin{lemma} \label{lem::oneLongCycle}
Let $k \geq 3$ be an integer and let $G_0 = (V,E)$ be a graph on $n
\geq 6$ vertices. Assume that there exists a non-trivial connected
component $C_1$ of $G_0$ such that $G_0[C_1]$ admits a Hamilton
cycle $C$ and that $d_{G_0}(u) = 0$ for every $u \in V \setminus
C_1$.
\begin{description}
\item [(a)] If $|C_1| = 2k-1$, then $s(G_0, \mathcal{M}_k) \leq \binom{2k-1}{2}$.
\item [(b)] If $|C_1| = 2k-2$ and $\binom{2k-2}{2} - e(G_0)$ is even, then $s(G_0, \mathcal{M}_k) \leq \binom{2k-1}{2}$.
\end{description}
\end{lemma}

\begin{proof}
Part (a) is trivial since, throughout the $(G_0, \mathcal{M}_k)$
game, the only legal edges are those with both endpoints in $C_1$.
Hence, $e(G) \leq \binom{2k-1}{2}$ will hold at the end of the game
no matter how Mini plays. As for (b), Mini plays according to the
following simple strategy.

\begin{description}
\item [Stage I:] Let $uv$ denote the last edge claimed by Max; we distinguish between the following two cases:
\begin{description}
\item [(1)] If $\{u,v\} \subseteq C_1$, then Mini claims an arbitrary free edge $w w'$ such that $\{w,w'\} \subseteq C_1$ and repeats Stage I.
\item [(2)] Otherwise, assume without loss of generality that $v \in C_1$ and $u \notin C_1$. Mini claims $uv'$, where $v'$ is a neighbor of $v$ in $C$. She then proceeds to Stage II.
\end{description}
\item [Stage II:] Throughout this stage, Mini follows the trivial strategy.
\end{description}

Since no edge $xy$ such that $\{x,y\} \in V \setminus C_1$ is legal,
it follows that the proposed strategy does account for every legal
move of Max. Moreover, since Mini is the second player and
$\binom{2k-2}{2} - e(G_0)$ is even, it follows that she can play
according to Case (1) of Stage I. Hence, at some point during the
game, Max must claim an edge $uv$ such that $|\{u,v\} \cap C_1| =
1$. By Case (2) of Stage I and by the analysis of Part (a) of the
lemma, we conclude that $e(G) \leq \binom{2k-1}{2}$ will hold at the
end of the game.
\end{proof}

\begin{lemma} \label{lem::edgeOrTriangle}
Let $k \geq 3$ be an integer and let $G_0 = (V,E)$ be a graph on $n
\geq 6$ vertices. Assume that there are two non-trivial connected
component $C_1$ and $C_2$ of $G_0$, where $|C_1| = 2k-3$. Assume
further that $G_0[C_1]$ admits a Hamilton cycle $C$ and that
$d_{G_0}(u) = 0$ for every $u \in V \setminus (C_1 \cup C_2)$.
\begin{description}
\item [(a)] If $G_0[C_2] \cong K_3$, then $s(G_0, \mathcal{M}_k) \leq \binom{2k-1}{2}$.
\item [(b)] If $G_0[C_2] \cong K_2$ and $\binom{2k-3}{2} - e(G_0[C_1])$ is even, then $s(G_0, \mathcal{M}_k) \leq \binom{2k-1}{2}$.
\end{description}
\end{lemma}

\begin{proof}
Part (a) is trivial since, throughout the $(G_0, \mathcal{M}_k)$
game, the only legal edges are those with both endpoints in $C_1$.
Hence, by following the trivial strategy, Mini ensures that $e(G)
\leq 3 + \binom{2k-3}{2} \leq \binom{2k-1}{2}$ will hold at the end
of the game. As for (b), Mini plays according to the following
simple strategy.

\begin{description}
\item [Stage I:] Let $uv$ denote the last edge claimed by Max; we distinguish between the following two cases:
\begin{description}
\item [(1)] If $\{u,v\} \subseteq C_1$, then Mini claims an arbitrary free edge $w w'$ such that $\{w,w'\} \subseteq C_1$ and repeats Stage I.
\item [(2)] Otherwise, assume without loss of generality that $u \in C_2$. Let $u'$ be the unique vertex in $C_2 \setminus \{u\}$. If $v \in C_1$, Mini claims $u' v'$, where $v'$ is a neighbor of $v$ in $C$. Otherwise, Mini claims $u' v$. In either case, she then proceeds to Stage II.
\end{description}
\item [Stage II:] Throughout this stage, Mini follows the trivial strategy.
\end{description}

Since every legal edge either has two endpoints in $C_1$ or one
endpoint in $C_2$, the proposed strategy does account for every
legal move of Max. Moreover, since Mini is the second player and
$\binom{2k-3}{2} - e(G_0[C_1])$ is even, it follows that she can
play according to Case (1) of Stage I. Hence, at some point during
the game, Max must claim an edge $uv$ such that $|\{u,v\} \cap C_1|
\leq 1$. If $|\{u,v\} \cap C_1| = 0$, then by Case (2) of Stage I
and by the analysis of Part (a) of the lemma, we conclude that $e(G)
\leq \binom{2k-1}{2}$ will hold at the end of the game. If $|\{u,v\}
\cap C_1| = 1$, then by Case (2) of Stage I and by
Lemma~\ref{lem::oneLongCycle}(a), we conclude that $e(G) \leq
\binom{2k-1}{2}$ will hold at the end of the game.
\end{proof}

\begin{lemma} \label{lem::edgeOrTriangle2}
Let $k \geq 4$ be an integer and let $G_0 = (V,E)$ be a graph on $n
\geq 8$ vertices. Assume that there are two non-trivial connected
component $C_1$ and $C_2$ of $G_0$, where $|C_1| = 2k-4$. Assume
further that $G_0[C_1]$ admits a Hamilton cycle $C$, that
$d_{G_0}(u) = 0$ for every $u \in V \setminus (C_1 \cup C_2)$, and
that $\binom{2k-4}{2} - e(G_0[C_1])$ is even.
\begin{description}
\item [(a)] If $G_0[C_2] \cong K_3$, then $s(G_0, \mathcal{M}_k) \leq \binom{2k-1}{2}$.
\item [(b)] If $G_0[C_2] \cong K_2$, then $s(G_0, \mathcal{M}_k) \leq \binom{2k-1}{2}$.
\end{description}
\end{lemma}

\begin{proof}
Starting with (a), Mini plays according to the following simple
strategy.

\begin{description}
\item [Stage I:] Let $uv$ denote the last edge claimed by Max; we distinguish between the following two cases:
\begin{description}
\item [(1)] If $\{u,v\} \subseteq C_1$, then Mini claims an arbitrary free edge $w w'$ such that $\{w,w'\} \subseteq C_1$ and repeats Stage I.
\item [(2)] Otherwise, assume without loss of generality that $v \in C_1$ and $u \notin C_1$. If $u \in C_2$, Mini claims $u' v'$, where $u'$ is some vertex of $C_2 \setminus \{u\}$ and $v'$ is a neighbor of $v$ in $C$. Otherwise, Mini claims $u v'$, where $v'$ is a neighbor of $v$ in $C$. In either case, she then proceeds to Stage II.
\end{description}
\item [Stage II:] Throughout this stage, Mini follows the trivial strategy.
\end{description}

Since every legal edge has at least one endpoint in $C_1$, it
follows that the proposed strategy does account for every legal move
of Max. Moreover, since Mini is the second player and
$\binom{2k-4}{2} - e(G_0[C_1])$ is even, it follows that she can
play according to Case (1) of Stage I. Hence, at some point during
the game, Max must claim an edge $uv$ such that $|\{u,v\} \cap C_1|
= 1$. If $|\{u,v\} \cap C_2| = 1$, then by Case (2) of Stage I and
by Lemma~\ref{lem::oneLongCycle}(a) and its proof, we conclude that
$e(G) \leq \binom{2k-1}{2}$ will hold at the end of the game.
Otherwise, by Case (2) of Stage I and by
Lemma~\ref{lem::edgeOrTriangle}(a) and its proof, we conclude that
$e(G) \leq \binom{2k-1}{2}$ will hold at the end of the game.

As for (b), Mini plays according to the following simple strategy.
Let $uv$ denote the last edge claimed by Max; we distinguish between
the following three cases:
\begin{description}
\item [(1)] If $\{u,v\} \subseteq C_1$, then Mini claims an arbitrary free edge $w w'$ such that $\{w,w'\} \subseteq C_1$.
\item [(2)] Otherwise, if $\{u,v\} \cap C_1 = \emptyset$, assume without loss of generality that $u \in C_2$ and $v \in V \setminus (C_1 \cup C_2)$. Mini claims $u' v$, where $u'$ is the unique vertex in $C_2 \setminus \{u\}$ and then follows the strategy described in the proof of Part (a) of the lemma until the end of the game.
\item [(3)] Otherwise, assume without loss of generality that $v \in C_1$ and let $v'$ be a neighbor of $v$ in $C$. If $u \in C_2$, Mini claims $u' v'$, where $u'$ is the unique vertex in $C_2 \setminus \{u\}$ and then follows the strategy described in the proof of Lemma~\ref{lem::oneLongCycle}(b) until the end of the game. Otherwise, Mini claims $u v'$ and then follows the strategy described in the proof of Lemma~\ref{lem::edgeOrTriangle}(b) until the end of the game.
\end{description}

Since every legal edge has at least one endpoint in $C_1 \cup C_2$,
it follows that the proposed strategy does account for every legal
move of Max. Moreover, since Mini is the second player and
$\binom{2k-4}{2} - e(G_0[C_1])$ is even, it follows that she can
play according to Case (1). Hence, at some point during the game,
Max must claim an edge $uv$ such that $|\{u,v\} \cap C_1| \leq 1$.
If $|\{u,v\} \cap C_1| = 0$, then by Case (2) and by Part (a) of the
lemma, we conclude that $e(G) \leq \binom{2k-1}{2}$ will hold at the
end of the game. Otherwise, by Case (3) and by
Lemmas~\ref{lem::oneLongCycle}(b) and~\ref{lem::edgeOrTriangle}(b),
we conclude that $e(G) \leq \binom{2k-1}{2}$ will hold at the end of
the game.
\end{proof}

\begin{lemma} \label{lem::CycleAndMaybePendingEdge}
Let $k \geq 3$ be an integer and let $G_0 = (V,E)$ be a graph on $n
\geq 6$ vertices. Assume that there exists a non-trivial connected
component $C_1$ of $G_0$ of order $2k-3$ such that $G_0[C_1]$ admits
a Hamilton cycle $C$. Let $x \in V \setminus C_1$ and assume that
$d_{G_0}(x) \leq 1$ and $d_{G_0}(u) = 0$ for every $u \in V
\setminus (C_1 \cup \{x\})$.
\begin{description}
\item [(a)] If $\binom{2k-2}{2} - e(G_0)$ is even and there exists a vertex $w \in C_1$ such that $wx \in E$, then $s(G_0, \mathcal{M}_k) \leq \binom{2k-1}{2}$.
\item [(b)] If $\binom{2k-3}{2} - e(G_0)$ is odd and $d_{G_0}(x) = 0$, then $s(G_0, \mathcal{M}_k) \leq \binom{2k-1}{2}$.
\end{description}
\end{lemma}

\begin{proof}
Starting with (a), Mini plays according to the following simple
strategy.

\begin{description}
\item [Stage I:] Let $uv$ denote the last edge claimed by Max; we distinguish between the following two cases:
\begin{description}
\item [(1)] If $\{u,v\} \subseteq C_1 \cup \{x\}$, then Mini claims a free edge $x w'$, where $w'$ is a neighbor of $w$ in $C$. She then follows the strategy described in the proof of Lemma~\ref{lem::oneLongCycle}(b) until the end of the game.
\item [(2)] Otherwise, assume without loss of generality that $d_{G_0}(u) = 0$. If $v = x$, Mini claims $u z$, where $z \in C_1$ is an arbitrary vertex and otherwise she claims $ux$. In either case, she then proceeds to Stage II.
\end{description}
\item [Stage II:] Throughout this stage, Mini follows the trivial strategy.
\end{description}

Since Mini is the second player and $\binom{2k-2}{2} - e(G_0)$ is
even, it follows that Mini can play according to Case (1) of Stage I
and thus, by Lemma~\ref{lem::oneLongCycle}(b), ensure that $e(G)
\leq \binom{2k-1}{2}$ will hold at the end of the game. If Mini
plays according to Case (2) of Stage I, then after her first move,
every legal edge has both endpoints in $C_1 \cup \{u,x\}$ and thus
$e(G) \leq \binom{2k-1}{2}$ will hold at the end of the game.

As for (b), Mini plays according to the following simple strategy.
Let $uv$ denote the last edge claimed by Max; we distinguish between
the following three cases:
\begin{description}
\item [(1)] If $\{u,v\} \subseteq C_1$, then Mini claims an arbitrary free edge $w w'$ such that $d_{G_0}(w) = d_{G_0}(w') = 0$ and then follows the strategy described in the proof of Lemma~\ref{lem::edgeOrTriangle}(b) until the end of the game.
\item [(2)] Otherwise, if $\{u,v\} \cap C_1 = \emptyset$, then Mini claims an arbitrary free edge $w w'$ such that $\{w,w'\} \subseteq C_1$ and then follows the strategy described in the proof of Lemma~\ref{lem::edgeOrTriangle}(b) until the end of the game.
\item [(3)] Otherwise, assume without loss of generality that $v \in C_1$ and $d_{G_0}(u) = 0$. Mini claims $u v'$, where $v'$ is a neighbor of $v$ in $C$ and then follows the strategy described in the proof of Lemma~\ref{lem::oneLongCycle}(b) until the end of the game.
\end{description}

Since $\binom{2k-3}{2} - e(G_0)$ is odd, it follows that Mini can
play according to the proposed strategy. Moreover, it follows by
Lemmas~\ref{lem::edgeOrTriangle}(b) and~\ref{lem::oneLongCycle}(b)
that $e(G) \leq \binom{2k-1}{2}$ will hold at the end of the game.
\end{proof}

We can now describe Mini's strategy for the $k$-matching saturation
game $(n, \mathcal{M}_k)$. At any point during the game, if Mini is
unable to follow the proposed strategy, then she forfeits the game.
The proposed strategy is divided into the following three stages.

\begin{description}
\item [Stage I:] Mini follows the long path strategy until $G$ contains a path $P = (u_0, \ldots, u_{\ell})$ of
length $\ell \in \{2k-5,2k-4\}$ which includes all vertices of
positive degree. At that moment, if $\ell = 2k-5$, then Mini
proceeds to Stage II, otherwise she skips to Stage III.
\item [Stage II:] Let $uv$ denote the edge Max claims in his subsequent move; we distinguish between the following two cases:
\begin{description}
\item [(1)] If $\{u,v\} \cap V(P) \neq \emptyset$, then Mini plays as follows.
If $\{u,v\} \subseteq V(P)$, then Mini claims $u_{\ell} z$ for an
arbitrary vertex $z \in V(G) \setminus V(P)$. Otherwise, assume
without loss of generality that $u \notin V(P)$. Mini then claims
$u_{\ell} u$ if it is free and $u_0 u$ otherwise. In either case she
extends $P$ to a path of length $2k-4$. By abuse of notation and for
simplicity of presentation, we denote this path by $P = (u_0, \ldots, u_{\ell})$ as
well. Mini then proceeds to Stage III.
\item [(2)] Otherwise, Mini claims the edge $u_0 u_\ell$, and then plays according to the strategy described in the proof of
Lemma~\ref{lem::edgeOrTriangle2}(b) until the end of the game.
\end{description}
\item [Stage III:] Let $uv$ denote the edge Max claims in his subsequent move. Mini claims $u_0 u_\ell$ if it is free and an arbitrary edge $w w'$ such that $\{w, w'\} \subseteq V(P)$ otherwise; we then distinguish between the following three cases:
\begin{description}
\item [(1)] If $|\{u,v\} \cap V(P)| = 0$, then Mini plays according to the strategy described in the proof of
Lemma~\ref{lem::edgeOrTriangle}(b) until the end of the game.
\item [(2)] If $|\{u,v\} \cap V(P)| = 1$, then Mini plays according to the strategy described in the proof of
Lemma~\ref{lem::CycleAndMaybePendingEdge}(a) until the end of the
game.
\item [(3)] If $|\{u,v\} \cap V(P)| = 2$, then Mini plays according to the strategy described in the proof of
Lemma~\ref{lem::CycleAndMaybePendingEdge}(b) until the end of the
game.
\end{description}
\end{description}

It follows by Lemma~\ref{lem::longPath} that Mini can play according
to Stage I of the proposed strategy. Lemma~\ref{lem::longPath} also
ensures that $u_0 u_{\ell}$ is free if Mini wishes to follow Case
(2) of Stage II (the only possible exception is the case $\ell = 1$,
but this can only occur if $k=3$ and Mini is the first player; this
case is excluded by our assumption that the parity of Mini is
opposite to the parity of $k$). Mini can play according to the
remainder of the proposed strategy by our assumption that the parity
of Mini is opposite to the parity of $k$.

Finally, it follows by Lemmas~\ref{lem::edgeOrTriangle2}(b),
\ref{lem::edgeOrTriangle}(b), \ref{lem::CycleAndMaybePendingEdge}(a)
and~\ref{lem::CycleAndMaybePendingEdge}(b) that $e(G) \leq
\binom{2k-1}{2}$ will hold at the end of the game. {\hfill $\Box$
\medskip\\}

\section{Concluding remarks and open problems}
\label{sec::openprob}

In this paper we proved lower and upper bounds on the scores of
several natural saturation games, namely, connectivity, colorability
and matching games. Other natural graph properties could be
considered; one interesting example is Hamiltonicity. Let ${\mathcal
H}$ denote the graph property of admitting a Hamilton cycle. It was
proved by Ore~\cite{Ore} that $ex(n, {\mathcal H}) = \binom{n-1}{2}
+ 1$. On the other hand, it is known (see, e.g.,~\cite{FFS}) that if
$n$ is not too small, then $sat(n, {\mathcal H}) = \lceil 3n/2
\rceil$. Our attempts to determine $s(n, {\mathcal H})$ lead us to
make the following conjecture.

\begin{conjecture} \label{conj::scoreHam}
$s(n, {\mathcal H}) = \Theta(n^2)$.
\end{conjecture}

All games considered in this paper require Max and Mini to avoid
certain large structures. Another interesting line of research would
be to avoid small structures. Given a fixed graph $H$, let
${\mathcal F}_H$ denote the graph property of admitting a copy of
$H$. It follows from the celebrated Stone-Erd\H{o}s-Simonovits
Theorem (see, e.g.,~\cite{BolBook}) that $ex(n, {\mathcal F}_H) =
\Theta(n^2)$ holds for every non-bipartite graph $H$. On the other
hand, it was proved by K\'aszonyi and Tuza~\cite{KT} that $sat(n,
{\mathcal F}_H) = O(n)$ for every graph $H$. As noted in the
introduction, very little is known about $s(n, {\mathcal F}_H)$,
even in the case $H = K_3$. Several simpler cases were considered
in~\cite{CKRW}.

For most graph properties ${\mathcal P}$ considered in this paper,
we have shown that the score of the $(n, {\mathcal P})$ saturation
game is very close to the trivial upper bound $ex(n, {\mathcal P})$.
A bold exception are the $k$-matching games under some assumptions
on the parity of $k$ and the identity of the first player. It is not
hard to find examples of properties ${\mathcal P}$ for which the
trivial lower bound $s(n, {\mathcal P}) \geq sat(n, {\mathcal P})$
is in fact tight. For example, as shown in
Theorem~\ref{th::kMatching}, if Mini is the first player, then $s(n,
{\mathcal M}_2) = 3 = sat(n, {\mathcal M}_2)$. In fact, there are
infinitely many such examples. For every integer $k \geq 2$, let
$\alpha_k$ denote the property of having independence number less
than $k$. If $G \in \alpha_k$ then clearly $G$ admits an independent
set $I$ of size $k$ and $uv \in E(G)$ whenever $\{u,v\} \setminus I
\neq \emptyset$. It follows that $sat(n, \alpha_k) = ex(n, \alpha_k)
= \binom{n}{2} - \binom{k}{2}$ and thus $s(n, \alpha_k)
= \binom{n}{2} - \binom{k}{2}$ as well. It would be interesting to find less
obvious examples of the tightness of the trivial lower bound.

\textbf{Acknowledgement.} We would like to thank the anonymous
referees for their useful comments and for pointing out an error in
an earlier version of this paper.

\end{document}